\documentclass[12pt]{amsart}
\usepackage{fullpage}
\usepackage{latexsym}
\usepackage{amssymb}
\usepackage[all]{xy}
 \makeindex

\newtheorem{thm}{Theorem}[section]
\newtheorem{prop}[thm]{Proposition}
\newtheorem{cor}[thm]{Corollary}

\newtheorem{lem}[thm]{Lemma}

\newtheorem{prob}[thm]{Problem}

\theoremstyle{definition}
\newtheorem{rem}[thm]{Remark}
\newtheorem{ex}[thm]{Example}
\newtheorem{defn}[thm]{Definition}

\usepackage{hyperref}

\newcommand{\CP}{\mathbb{CP}}
\newcommand{\N}{\mathbb{N}}
\newcommand{\Q}{\mathbb{Q}}
\newcommand{\R}{\mathbb{R}}
\newcommand{\Z}{\mathbb{Z}}
\newcommand{\RP}{\mathbb{RP}}

\newcommand{\SO}{\mathrm{\SO}}
\newcommand{\SL}{\mathrm{SL}}
\newcommand{\Nil}{\mathrm{Nil}}

\title[{\tiny Geometric structures in Topology, Geometry, Global Analysis and Dynamics}]{Geometric structures in Topology, Geometry,\\ Global Analysis and Dynamics}

\author{Christoforos Neofytidis}
\address{Department of Mathematics, Ohio State University, Columbus, OH, USA}
\email{neofytidis.1@osu.edu}
\date{\today}

\begin{document}

\maketitle

\begin{abstract}
Following Thurston's geometrisation picture in dimension three, we study 
 geometric manifolds in a more general setting in arbitrary dimensions, with respect to the following problems: (i) The existence of maps of non-zero degree (domination relation or Gromov's order); (ii) The Gromov-Thurston monotonicity problem for numerical homotopy invariants with respect to the domination relation; (iii) The existence of Anosov diffeomorphisms (Anosov-Smale conjecture).
\end{abstract}

\section{Introduction}

Thurston's work has initiated and motivated tremendous research activity in various directions. The purpose of this survey is to present how Thurston's geometrisation picture for $3$-manifolds can be used and extended in high dimensions, including both geometric manifolds in the sense of Thurston and other non-geometric manifolds, to give a unified treatment of a diversity of problems arising in Topology, Geometry, Global Analysis and Dynamics. 

At the topological level, we will be dealing with an ordering of homotopy classes of manifolds of a given dimension, called {\em Gromov order}\index{Gromov!order} or {\em domination relation}\index{domination}. We shall say that a manifold $M$ {\em dominates} $N$, and write $M\geq N$, if there is a map $f\colon M\to N$ of non-zero degree. The domination relation has been studied by many people and in various contexts, using a plethora of techniques and tools, such as cohomology ring structures and intersection forms in Algebraic Topology, 
 bounded cohomology and the simplicial volume in Geometry and Global Analysis, 
 the fundamental group in Group Theory, as well as the theory of harmonic mappings in Complex and Harmonic Analysis. In this survey, we will present an ordering  in the sense of S. Wang~\cite{Wangorder} of all $3$-manifolds and geometric $4$-manifolds.

As indicated above, the simplicial volume $\|\cdot\|$ is a significant tool in the study of the domination relation, and it is an example of a {\em functorial semi-norm} in homology. Namely, if $f\colon M\to N$ is a map of degree $d$, then $\|M\|\geq|d|\|N\|$. In dimension two, this can be restated in terms of the absolute Euler characteristic $|\chi|$. Pointing out again Thurston's influence, we quote the following from Gromov's book {\em Metric Structures for Riemannian and Non-Remannian spaces}~\cite[pg. 300]{Gromovbook}: {\em ``The interpretation of $|\chi|$ as a norm originates from a work by Thurston, who used this idea to define a norm on $H_2(X^3)$ using surfaces embedded into $3$-manifolds"}. From this more geometric and global analytic point of view, our second goal in this survey is to study the following  monotonicity problem: {\em Given a numerical invariant $\iota$, does $M\geq N$ imply $\iota(M)\geq\iota(N)$?} We will introduce a notion of {\em geometric Kodaira dimension $\kappa^g$} and show that $M\geq N$ implies $\kappa^g(M)\geq\kappa^g(N)$ for all $3$-manifolds and geometric manifolds in dimensions four and five. We will compare our definition of $\kappa^g$ with traditional notions of Kodaira dimension in Complex Geometry and establish relations to the simplicial volume.

The last part of this survey has a dynamical flavor, namely the study of Anosov diffeomorphisms. A long-standing conjecture, going back to Anosov and Smale, asserts that all Anosov diffeomorphisms are conjugate to hyperbolic automorphisms of nilmanifolds~\cite{Sm}. Algebraic tools, such as Hirsch and Ruelle-Sullivan cohomology classes, coarse geometric methods, such as negative curvature, and many other techniques from various areas have been proven fruitful in understanding the Anosov-Smale conjecture; see for example~\cite{GH},~\cite{GL},~\cite{NeoAnosov1} and their references. Here, with a more unified approach achieved via Thurston's geometries, we will explain how to rule out Anosov diffeomorphisms from all Thurston geometric $4$-manifolds that are not covered by the product of two surfaces of positive genus. 

Throughout this survey (and for the sake of simplicity) all manifolds are assumed to be closed, oriented and connected.

\subsection*{Acknowledgements}
I would like to thank Ken’ichi Ohshika and Athanase Papadopoulos for their invitation to write this survey, as well as an anonymous referee for the careful reading and the suggestions.

\section{Domination, Monotonicity and Anosov maps}\label{s:preliminary}

We begin our discussion with an overview of the three main topics of this survey, indicating as well their state of the art with various open questions. This section aims also to
 serve as an introduction for readers not familiar with these topics.
 
\subsection{The domination relation}

\begin{defn}\label{d:domination}
Let $M,N$ be manifolds  of dimension $n$. We say that $M$ {\em dominates} $N$ if there is a map $f\colon M\to N$ of non-zero degree $d$. We denote this by $M\geq N$ or by $M\geq_d N$ when we need to emphasise on the specific degree $d$. We also write $\deg(f)$ for the degree.
\end{defn}

Recall that $f\colon M\to N$ being of degree $d$ means that the induced map in homology $H_n(f)\colon H_n(M;\Z)\to H_n(N;\Z)$ satisfies $H_n(f)([M])=d\cdot[N]$, where $[X]$ denotes the fundamental (orientation) class\index{non-zero degree}.

\medskip

Domination is a transitive relation. For, if $M\geq N$ and $N\geq W$, then $M\geq W$ via composition of the two dominant maps. In a lecture at CUNY Graduate Center in 1978, Gromov suggested studying the domination relation as an ordering of the homotopy classes of manifolds of the same dimension (hence the name {\em Gromov order})\index{Gromov!order}. In dimension two, the domination relation is indeed a total order given by the genus:

\begin{prop}\label{p:surfaceorder}
Let $\Sigma_g, \Sigma_h$ be two surfaces of genus $g$ and $h$ respectively. Then, $\Sigma_g\geq \Sigma_h$ if and only if $g\geq h$.
\end{prop}
\begin{proof}
For the ``if" part, we observe that for each $g$, there is a degree one map, called {\em pinch map}, given by
\begin{equation}\label{pinch}
\Sigma_g\cong\underbrace{T^2\#\cdots\#T^2}_{g}\longrightarrow\underbrace{T^2\#\cdots\#T^2}_{g-1}\vee T^2\longrightarrow\underbrace{T^2\#\cdots\#T^2}_{g-1}\vee pt\cong\Sigma_{g-1}.
\end{equation}
The ``only if" part follows by the next more general lemma, since $H_1(\Sigma_g)\cong\Z^{2g}$.
\end{proof}

\begin{lem}\label{l:Betti}
If $M\geq N$, then $b_i(M)\geq b_i(N)$, where $b_i(X)=\dim H_i(X;\Q)$ denotes the $i$-th Betti number of $X$.
\end{lem}
\begin{proof}
Clearly it suffices to prove the lemma for $0<i<n$. Let $f\colon M\to N$ be a map of non-zero degree $d$ and $\alpha\in H_i(N;\Q)$. Consider the preimage under the Poincar\'e duality isomorphism $PD_N^{-1}(\alpha)\in H^{n-i}(N;\Q)$ and then the image $H^{n-i}(f)(PD_N^{-1}(\alpha))\in H^{n-i}(M;\Q)$. Let the homology class
\[
\beta:=PD_MH^{n-i}(f)PD_N^{-1}(\alpha)=H^{n-i}(f)(PD_N^{-1}(\alpha))\cap[M]\in H_{i}(M;\Q).
\]
Then we obtain
\[
\begin{aligned}
H_i(f)(\beta)&=H_i(f)(H^{n-i}(f)(PD_N^{-1}(\alpha))\cap[M])\\
&=PD_N^{-1}(\alpha)\cap H_i(f)([M])\\
&=d\cdot PD_N^{-1}(\alpha)\cap[N]\\
&=d\cdot PD_N(PD_N^{-1}(\alpha))\\
&=d\cdot\alpha.
\end{aligned}
\]
That is, 
\[
H_i(f)\biggl(\frac{1}{d}\cdot\beta\biggl)=\alpha,
\]
which means that $H_i(f)$ with rational coefficients is surjective and the lemma follows.
\end{proof}

In higher dimensions, the domination relation is not anymore a total ordering as the following example shows:

\begin{ex}\label{ex:non-order}
The 3-sphere $S^3$ is a 2-fold cover of the projective plane $\RP^3=S^3/\Z_2$. For any $n$-manifold $M$, there is a degree one pinch map $M\cong M\#S^n\to S^n$ (as in Proposition \ref{p:surfaceorder}; see \eqref{pinch}). Hence, 
\begin{equation}\label{eq:non-total}
S^3\geq_2\RP^3\geq_1 S^3,
\end{equation}
while of course $S^3$ and $\RP^3$ are not homotopy equivalent.
\end{ex}

Naturally, the dominations given by \eqref{eq:non-total} raise the following:

\begin{prob}\label{p:Hopfgeneral}
Suppose that $M\geq_1 N\geq_1 M$. Are $M$ and $N$ homotopy equivalent?
\end{prob}

This is tightly related to the next long-standing problem of Hopf\index{Hopf!problem}:

\begin{prob}\cite[Problem 5.26]{Kirby}\label{p:Hopf5.26}
Is every self map of degree $\pm1$ a homotopy equivalence?
\end{prob}

At the group theoretic level one has the following corresponding concept:

\begin{defn}\label{d:Hopf}
A group $G$ is called {\em Hopfian}\index{group!Hopfian} \index{Hopfian group} if every surjective endomorphism of $G$ is an isomorphism.
\end{defn}

An affirmative answer to Problem \ref{p:Hopf5.26} holds for the class of aspherical manifolds with Hopfian fundamental groups. Recall that a 
manifold $M$ is called {\em aspherical}\index{aspherical}\index{manifold!aspherical} if all its homotopy groups $\pi_k(M)$ vanish for $k\geq2$.

\begin{prop}\label{p:asphericalHopf}
Let $M$ be an aspherical manifold with Hopfian fundamental group $\pi_1(M)$. Then every map $f\colon M\to M$ with $\deg(f)=\pm1$ is a homotopy equivalence.
\end{prop}
\begin{proof}
We begin our proof by recalling the following well-known lemma:

\begin{lem}\label{l:degreegroups}
Let $M,N$ be manifolds of the same dimension. If $f\colon M\to N$ is a map of non-zero degree, then $[\pi_1(N):f_*(\pi_1(M))]<\infty$, where $f_*\colon\pi_1(M)\to\pi_1(N)$ denotes the induced homomorphism. If, moreover, $\deg(f)=\pm1$, then $[\pi_1(N):f_*(\pi_1(M))]=1$.
\end{lem}
\begin{proof}
Let $\overline{N}\xrightarrow{p}N$ be the covering of degree $\deg(p)=[\pi_1(N):f_*(\pi_1(M))]$, which corresponds to $f_*(\pi_1(M))$. We then lift $f$ to $\overline f\colon M\to\overline N$, and we have $f=p\circ\overline f$. In particular, $\deg(f)=\deg(p)\deg(\overline f)$, which verifies both claims of the lemma.
\end{proof}

Since $f\colon M\to M$ has $\deg(f)=\pm1$, Lemma \ref{l:degreegroups} tells us that $f_*\colon\pi_1(M)\to\pi_1(M)$ is surjective. By assumption $\pi_1(M)$ is Hopfian, hence $f_*$ is an isomorphism. The proposition now follows by Whitehead's classical theorem (see for example~\cite[Theorem 4.5]{Ha}), since $\pi_k(M)=0$ for all $k\geq2$.
\end{proof}

In particular, we obtain the following partial ordering:

\begin{cor}
The domination relation\index{domination}\index{Gromov!order} $\geq_1$ is a partial ordering on the class of aspherical manifolds with Hopfian fundamental groups.
\end{cor}

Note that the requirement on $\pi_1$ being Hopfian might be redundant:

\begin{prob}\label{pr:Hopfaspherical}
Is the fundamental group of any aspherical manifold Hopfian?
\end{prob}

Problem \ref{pr:Hopfaspherical} has a complete affirmative answer in dimensions $\leq 3$ and in all other known cases in higher dimensions, such as for nilpotent manifolds. Finally, concerning the case of self-maps, we have the following strong version of Problem \ref{p:Hopf5.26} for aspherical manifolds:

\begin{prob}\cite[Problem 1.2]{NeoHopf}\label{pr:Hopfasphericalstrong}
Is every self-map of an aspherical manifold either a homotopy equivalence (when the degree is $\pm1$) or homotopic to a non-trivial covering?
\end{prob}

If moreover ``homotopy equivalence" is replaced by ``homotopic to a homeomorphism" (in other words, {\em is every self-map of an aspherical manifold homotopic to a covering?}), then Problem \ref{pr:Hopfasphericalstrong} becomes a strong version of the Borel conjecture\index{conjecture!Borel}, which asserts that any homotopy equivalence between closed aspherical manifolds is homotopy to a homeomorphism.

\subsection{Monotone invariants}

\begin{defn}\label{d:monotone}
Let $M$ be a manifold. A non-negative numerical quantity $\iota(M)$ is {\em monotone}\index{invariant!monotone} with respect to the domination relation if
\[
M\geq N \Rightarrow \iota(M)\geq\iota(N).
\]
\end{defn}

Clearly such a number is a homotopy invariant. If one requires furthermore the degree\index{non-zero degree} of the map to be carried in the inequality, i.e.,
\[
M\geq_d N \Rightarrow \iota(M)\geq|d|\iota(N),
\]
then we say that $\iota$ is {\em functorial}\index{semi-norm!functorial}. Amongst the most prominent functorial homotopy invariants is the simplicial volume\index{simplicial volume}.

\begin{defn}\label{d:simplicialv}
Given a topological space $X$ and a homology class $\alpha\in H_n(X;\R)$, the {\em Gromov norm} of $\alpha$ is defined to be
\[
\|\alpha\|_1:=\inf \biggl\{\sum_{j} |\lambda_j| \ \biggl\vert  \ \sum_j \lambda_j\sigma_j\in C_n(X;\R) \ \text{is a singular cycle representing } \alpha  \biggl\}.
\]
If, moreover, $X$ is an $n$-manifold, then the Gromov norm or {\em simplicial volume} of $X$ is given by $\|X\|:=\|[X]\|_1$.
\end{defn}

The functoriality  of the simplicial volume follows easily by the above definition:

\begin{lem}\label{l:monotonesimplicial}
Let $f\colon X\to Y$ be a map between topological spaces. Then $\|\alpha\|_1\geq\|H_n(f)(\alpha)\|_1$ for any $\alpha\in H_n(X;\R)$. In particular, if $M\geq_d N$, then $\|M\|\geq|d|\|N\|$.
\end{lem}

There is a tight connection between domination\index{domination} and monotone invariants. For if $M\geq_d M$ for $d>1$, then $\iota(M)=0$ for all {\em finite} functorial invariants $\iota$. Equivalently, the existence of a finite non-zero functorial invariant on $M$ implies that $M$ does not admit any self-maps of degree other than $0$ and $\pm1$. We record some (not necessarily mutually disjoint) examples regarding the simplicial volume:

\begin{ex}\label{ex:sv}\
\begin{itemize}
\item[(1)] The following classes of manifolds have zero simplicial volume: (a) spheres; (b) {\em rationally inessential manifolds}, i.e., manifolds $M$ whose classifying map $c_M\colon M\to B\pi_1(M)$ vanishes in top degree rational homology~\cite{Gromov}; (c) fiber bundles $F\to M\to B$, where $\pi_1(F)$ is Abelian or, more generally, amenable~\cite{Gromov}; (d) products with at least one factor with vanishing simplicial volume~\cite{Gromov}.
\item[(2)] The following classes of manifolds have non-zero simplicial volume: (a) hyperbolic manifolds~\cite{Gromov}; (b) irreducible, locally symmetric spaces of non-compact type~\cite{LS,Bucher}; (c) rationally essential manifolds with hyperbolic fundamental groups~\cite{Gromovessay,Mineyev}; (d) products whose all factors have positive simplicial volume~\cite{Gromov}.
\end{itemize}
\end{ex}

The manifolds in Example \ref{ex:sv} (1d) (resp. (2d)) have zero (resp. non-zero) simplicial volume
 because of the inequalities 
\begin{equation}\label{eq:simplicialproducts}
\|M\|\|N\|\leq\|M\times N\|\leq\left(\begin{array}{c}
   m+n\\
   n\\
 \end{array} \right)\|M\|\|N\|,
\end{equation}
where $m$ and $n$ denote the dimensions of $M$ and $N$ respectively.

Finally, we remark that the study of mapping degree sets\index{set!of mapping degrees} leads to knowledge about non-zero functorial invariants.

\begin{defn}\label{d:degreesets}
Let $M, N$ be $n$-manifolds. The {\em set of degrees of maps} from $M$ to $N$ is defined by
\[
D(M,N):=\{d\in\Z \ | \ d=\deg(f), \ f\colon M\to N\}.
\]
\end{defn}

Fixing $N$, one can define a functorial invariant\index{semi-norm!functorial} by looking at the supremum of all possible absolute degrees of maps to $N$ (see~\cite{CL}):
\begin{equation}\label{eq:mdsemi-norm}
\iota_N(M):=\sup\{|d| \ | \ d\in D(M,N)\}.
\end{equation}
If $D(N,N)\subseteq\{-1,0,1\}$, then clearly $\iota_N(N)=1$. In particular, one obtains non-vanishing functorial invariants on manifolds, where classical functorial invariants are known to be zero. A prominent class given in~\cite{NeoHopf} is that of not virtually (i.e., not finitely covered by) trivial $S^1$-bundles over hyperbolic manifolds for which the simplicial volume indeed vanishes~\cite{Gromov}. 

However, a disadvantage of \eqref{eq:mdsemi-norm} is that $\iota_N(\cdot)$ might not be finite, since the inclusion $D(N,N)\subseteq\{-1,0,1\}$ does not preclude the existence of a manifold $M$ so that $D(M,N)$ is unbounded. It is unknown whether this is the case for the not virtually trivial $S^1$-bundles over hyperbolic manifolds mentioned above, except in dimension three, where Brooks and Goldman~\cite{BG} showed the existence of another non-zero functorial invariant on $\widetilde{SL_2}$-manifolds, namely of the Seifert volume. 

\subsection{Anosov diffeomorphisms}

\begin{defn}\label{d:Anosov}
Suppose $M$ is a smooth $n$-manifold. A diffeomorphism $f\colon M\to M$ is called {\em Anosov}\index{Anosov diffeomorphism}\index{diffeomorphism!Anosov} if there exists a $df$-invariant splitting $TM=E^s\oplus E^u$ of the tangent bundle of $M$, together with constants $\mu\in (0,1)$ and $C > 0$, such that 
\begin{equation*}
\begin{split}
\|df^m(v)\| \leq C\mu^m\|v\|,\ \text{if} \ v \in E^s,\\  
\|df^{m}(v)\| \leq C^{-1}\mu^{-m}\|v\|,\ \text{if} \ v \in E^u,
\end{split}
\end{equation*}
for all $m\in\N$. 
\end{defn}

The invariant distributions $E^s$ and $E^u$ are called {\em stable} and {\em unstable} distributions respectively. An Anosov diffeomorphism $f$ is said to be of {\em codimension $k$}\index{Anosov diffeomorphism!codimension} if $E^s$ or $E^u$ has dimension $k \leq [n/2]$, and {\em transitive}\index{Anosov diffeomorphism!transitive} if there exists a point whose orbit is dense in $M$.

Currently, the only known examples of Anosov diffeomorphisms are of algebraic nature, namely, Anosov automorphisms of manifolds covered by nilmanifolds; see for example~\cite{Mal} about the interpretation of Anosov diffeomorphisms of nilmanifolds at the level of hyperbolic automorphisms of their fundamental groups. We illustrate this with two examples, one in the Abelian case and another one for a nilpotent but not Abelian group:

\begin{ex}\label{ex:AnosovAbelian}
The matrix \[
A=\left(\begin{array}{cc}
   2 & 1\\
   1 & 1\\
 \end{array} \right)
\]
has no eigenvalues which are roots of unity and hence it defines a hyperbolic automorphism of $\Z^2$.
\end{ex}

\begin{ex}\label{ex:AnosovNil}
Let 
\[
G=\langle x_1,x_2,...,x_6 \ |  \ [x_3,x_1]=x_5, \  [x_4,x_1]=[x_3,x_2]=x_6, \ [x_4,x_2]=x_5^3\rangle.
\]
This is a $6$-dimensional torsion-free, $2$-step nilpotent group. Indeed, the lower central series of $G$ is given by
\[
c_0(G)=G, \ c_1(G)=[c_0(G),G]=[G,G]=\langle x_5,x_6\rangle, \ c_2(G)=[c_1(G),G]=1.
\]
In particular, the quotient of $G$ by the isolator subgroup 
\[
\sqrt[G]{c_1(G)}=\{x\in G \ | \ x^k\in c_1(G) \ \text{for some integer} \ k\geq0\}= c_1(G)\cong\Z^2
\]
is isomorphic to $\Z^4\cong G/c_1(G)\cong\langle \overline{x_1},\overline{x_2},\overline{x_3},\overline{x_4}\rangle$.

By~\cite{CKS} and~\cite{Mal}, the group $G$ admits a hyperbolic automorphism. An explicit example is given in \cite[Example 3.5]{Mal}, namely, the automorphism $\phi\colon G\to G$ defined by
\[
x_1\mapsto x_1^2x_2^{-1}, \ x_2\mapsto x_1^{-3}x_2^{2}, \ x_3\mapsto x_3^7x_4^{4}, \  x_4\mapsto x_3^{12}x_4^{7}, \ x_5\mapsto x_5^2x_6^{1}, \  x_6\mapsto x_5^{3}x_6^{2}.
\]
Indeed, the restriction of $\phi$ to $c_1(G)$ is given by
\[
\left(\begin{array}{cc}
   2 & 3\\
   1 & 2\\
 \end{array} \right),
 \]
 and the induced automorphism $\overline\phi$ on $G/c_1(G)$ is given by
 \[
 \left(\begin{array}{rrrr}
   2 & -3 & 0 &0\\
   -1 & 2 & 0& 0\\
   0 & 0 & 7& 12 \\
   0 & 0 & 4& 7\\
 \end{array} \right).
 \]
 Both matrices define hyperbolic automorphisms (on $\Z^2$ and $\Z^4$ respectively), since they do not have eigenvalues which are roots of unity.
\end{ex}

Anosov and Smale~\cite{Sm} conjectured\index{conjecture!Anosov-Smale} that any Anosov diffeomorphism 
 is finitely covered by a diffeomorphism which is topologically conjugate to a hyperbolic automorphism of a nilpotent manifold. 
 
For the $m$-torus $T^m$, $m\geq2$, Franks~\cite{Fr} proved that if $f\colon T^m\to T^m$ is an Anosov diffeomorphism, then the induced isomorphism
\[
H_1(f)\colon H_1(T^m;\R)\longrightarrow H_1(T^m;\R)
\]
has no roots of unity among its eigenvalues. Hirsch~\cite{Hirsch} extended Franks’ result to Anosov diffeomorphisms of a wider class of manifolds which includes all nilmanifolds:

\begin{thm}\cite[Theorem 4]{Hirsch}\label{t:Hirsch4} 
Let $M$ be a manifold with virtually polycyclic fundamental group, whose universal covering has finite dimensional rational homology. If $M$ admits an Anosov diffeomorphism $f\colon M\to M$, then the isomorphism
\begin{equation*}\label{eq.eigenvalue}
H_1(f)\colon H_1(M;\R)\longrightarrow H_1(M;\R)
\end{equation*}
has no roots of unity among its eigenvalues.
\end{thm}

Hirsch's result has remarkable consequences, for instance, on polycyclic manifolds with infinite cyclic first integral cohomology group. In particular, mapping tori of Anosov diffeomorphisms do not themselves admit Anosov diffeomorphisms. Indeed, if
\[
M_A=T^n\rtimes_{A}S^1=\frac{T^n\times[0,1]}{(x,0)\sim(A(x),1)}
\]
is a mapping torus, such that none of the eigenvalues of $A\in\SL(n;\Z)$ is a root of unity, then $H^1(M_A;\Z)\cong H_1(M_A;\Z)/\mathrm{Tor}H_1(M_A;\Z)\cong\Z$. 

Ruelle-Sullivan \cite{RS} found an interesting obstruction related to the codimension of an Anosov diffeomorphism:

\begin{thm}\cite[Corollary pg. 326]{RS}\label{t:RS}
If $f\colon M\to M$ is a codimension\index{Anosov diffeomorphism!codimension} $k$ transitive Anosov diffeomorphism\index{Anosov diffeomorphism!transitive} with orientable invariant distributions, then there is a non-trivial cohomology class $\alpha\in H^k(M;\R)$ and a positive real number $\lambda\neq 1$ such that $H^k(f)(\alpha) = \lambda\cdot\alpha$. In particular,  $H^k(M;\R)\neq 0$.
\end{thm}

\section{The Gromov order for Thurston geometries in dimensions $\leq4$}

As explained in Example \ref{ex:non-order}, the domination relation in dimensions greater than two does not define an ordering of all manifolds in the usual sense. We thus need to find an alternative natural and meaningful method to order manifolds. 
In dimension three, such a method was proposed by S. Wang~\cite{Wangorder} following Thurston's geometrisation picture. We will first review Wang's ordering of Thurston's geometries, together with an extension of it to all 3-manifolds~\cite{KN}, and then describe an ordering of the 4-dimensional aspherical geometries. Our main reference is~\cite{Neoorder}.

\subsection{Classification of Thurston's geometries}\label{ss:Thurstonclassification}
We begin our discussion by recalling briefly the classification of Thurston geometries\index{Thurston geometry} in dimensions $\leq4$, together with some properties that we will need in our proofs.

Suppose $\mathbb{X}^n$ is a complete simply connected Riemannian manifold of dimension $n$. We will say that a manifold\index{manifold!geometric} $M$ is an {\em $\mathbb{X}^n$ manifold}, or is {\em modeled on $\mathbb{X}^n$}, or {\em carries the $\mathbb{X}^n$ geometry} in the sense of Thurston, if it is diffeomorphic to a quotient of
$\mathbb{X}^n$ by a lattice $\Gamma$ in the group of isometries of $\mathbb{X}^n$ 
(where $\Gamma=\pi_1(M)$). 
Two geometries $\mathbb{X}^n$ and $\mathbb{Y}^n$ are the same whenever there exists a diffeomorphism $\psi \colon \mathbb{X}^n
\longrightarrow \mathbb{Y}^n$ and an isomorphism $\mathrm{Isom}(\mathbb{X}^n) \longrightarrow \mathrm{Isom}(\mathbb{Y}^n)$ which sends each $g \in \mathrm{Isom}(\mathbb{X}^n)$ to $\psi \circ g \circ \psi^{-1} \in \mathrm{Isom}(\mathbb{Y}^n)$.

\subsection*{Dimension one}

The circle $S^1=\R/\Z$ is the only $1$-dimensional manifold and is modeled on $\R$.

\subsection*{Dimension two} 

Surfaces $\Sigma_g$, $g\geq0$, have been already discussed in Section \ref{s:preliminary}: For $g=0$ we have the 2-sphere $\Sigma_0=S^2$ (modeled on $S^2$), for $g=1$ the 2-torus $\Sigma_1=T^2=\R^2/\Z^2$ (modeled on $\R^2$) and for $g\geq2$ hyperbolic surfaces $\Sigma_g=\mathbb{H}^2/\pi_1(\Sigma_g)$ (modeled on $\mathbb{H}^2$), where
\[
\pi_1(\Sigma_g)=\langle a_1,b_1,...,a_g,b_g \ | \ [a_1,b_1]\cdots[a_g,b_g]=1\rangle.
\]
Table \ref{table:2geom} summarises the geometries in dimension two.
\begin{table}[!ht]
\centering
{\small
\begin{tabular}{c|c}
Type & Geometry $\mathbb{X}^2$\\
\hline
         Spherical   & $S^2$\\
         Euclidean    & $\R^2$ \\
Hyperbolic & $\mathbb{H}^2$\\            
\end{tabular}}
\vspace{9pt}
\caption{{\small The $2$-dimensional Thurston geometries}}\label{table:2geom}
\end{table}

\subsection*{Dimension three}

Thurston proved that there exist eight homotopically unique geometries: $\mathbb{H}^3$, $Sol^3$,
$\widetilde{SL_2}$, $\mathbb{H}^2 \times \R$, $Nil^3$, $\R^3$, $S^2 \times \R$ and $S^3$. In Table \ref{table:3geom}, we list the finite covers for manifolds in each of those geometries (see~\cite{Thu,Scott,Agol}). 

\begin{table}[!ht]
\centering
{\small
\begin{tabular}{r|l}
Geometry $\mathbb{X}^3$ & $M$ is finitely covered by...\\
\hline
$\mathbb{H}^3$     & a mapping torus of a hyperbolic surface with pseudo-Anosov monodromy\\
$Sol^3$            & a mapping torus of $T^2$ with hyperbolic monodromy\\
$\widetilde{SL_2}$ & a non-trivial $S^1$ bundle over a hyperbolic surface\\  
$Nil^3$            & a non-trivial $S^1$ bundle over $T^2$\\
$\mathbb{H}^2 \times \R$ & a product of $S^1$ with a hyperbolic surface\\
$\R^3$             & the $3$-torus $T^3$\\
$S^2 \times \R$    &  the product $S^2 \times S^1$\\
$S^3$              & the $3$-sphere $S^3$
\end{tabular}}
\newline
\caption{{\small Finite covers of Thurston geometric 3-manifolds.}}\label{table:3geom}
\end{table}

\subsection*{Dimension four}
The $4$-dimensional Thurston's geometries were classified by Filipkiewicz in his thesis~\cite{Filipkiewicz}. In Table \ref{table:4geom}, we list the geometries that are realised by compact manifolds, following~\cite{Wall1,Wall2} and~\cite{Hillman}. In the remainder of this paragraph we will mainly concentrate on the aspherical geometries.

\begin{table}[!ht]
\centering
{\small
\begin{tabular}{r|l}
Type  & Geometry $\mathbb{X}^4$\\
\hline
Hyperbolic & $\mathbb{H}^4$, $\mathbb{H}^2(\mathbb{C})$\\        
                Solvable  & $Nil^4$, 
$Sol^4_{m \neq n}$, $Sol^4_0$,
             $Sol^4_1$, $Sol^3\times\R$, $Nil^3\times\mathbb{R}$,  $\R^4$\\
             Compact &  $S^4$, $\mathbb{CP}^2$, $S^2\times S^2$\\
                 Mixed  products  &  $S^2\times\mathbb{H}^2$, $S^2\times \R^2$,  $S^3 \times \R$, 
 $\mathbb{H}^3\times\mathbb{R}$,
           $\mathbb{H}^2\times\mathbb{R}^2$, $\mathbb{H}^2\times\mathbb{H}^2$,
              $\widetilde{SL_2}\times\mathbb{R}$ \\
\end{tabular}}
\newline
\caption{{\small The 4-dimensional Thurston geometries with compact representatives.}}\label{table:4geom}
\end{table}

Manifolds modeled on a geometry of type $\mathbb{X}^3 \times \R$ satisfy the following property:

\begin{thm}\cite[Sections 8.5 and 9.2]{Hillman}\label{t:hillmanproducts}
 Let $\mathbb{X}^3$ be a $3$-dimensional 
 geometry. A $4$-manifold that carries the geometry $\mathbb{X}^3 \times \R$ is finitely
covered by a product $N \times S^1$, where $N$ is a 
$3$-manifold modeled on $\mathbb{X}^3$.
\end{thm}

Manifolds modeled on the geometry $\mathbb{H}^2 \times \mathbb{H}^2$ are either virtual products of two hyperbolic
surfaces or not even (virtual) surface bundles. These two types are distinguished by the names {\em reducible} and {\em irreducible}
$\mathbb{H}^2 \times \mathbb{H}^2$ geometry respectively; see~\cite[Section 9.5]{Hillman} for further details.

A class of 4-dimensional geometries that motivates some new phenomena with respect to the domination problem, especially the property {\em group (infinite-index) presentable by products} (see Definition~\ref{d:PP} and Proposition~\ref{p:sol&sol} below, as well as Section~\ref{monotonicityKodaira}) is that of solvable non-product geometries $Nil^4$, 
$Sol^4_{m \neq n}$, $Sol^4_0$ and $Sol^4_1$. Let us first recall the model Lie groups of those geometries together with some properties. 

The nilpotent Lie group $Nil^4$ is  the semi-direct product $\R^3 \rtimes \R$, where $\R$ acts on $\R^3$ by
\[
t \mapsto 
\left(\begin{array}{ccc}
   1 & e^t & 0 \\
   0 & 1 & e^t \\
   0 & 0 & 1   \\
\end{array} \right).
\]

\begin{prop}\cite[Prop. 6.10]{NeoIIPP}\label{p:nil4}
A $Nil^4$ manifold $M$ is finitely covered by a non-trivial $S^1$ bundle over a $Nil^3$ manifold and the center of $\pi_1(M)$ remains infinite cyclic in finite covers.
\end{prop}

Next, we give the model spaces for the three non-product solvable -- but not nilpotent -- geometries:
Suppose $m$ and $n$ are positive integers and $a > b > c$ are reals such that $a+b+c=0$ and $e^a,e^b,e^c$ are
roots for the polynomial $P_{m,n}(\lambda)=\lambda^3-m\lambda^2+n\lambda-1$. If $m \neq n$, the Lie group $Sol_{m \neq n}^4$ is a semi-direct product $\R^3 \rtimes
\R$, where $\R$ acts on $\R^3$ by
\[
t \mapsto 
\left(\begin{array}{ccc}
   e^{at} & 0 & 0 \\
   0 & e^{bt} & 0 \\
   0 & 0 & e^{ct} \\
\end{array} \right).
\]
Note that, when $m=n$, then $b = 0$ and this corresponds to the product geometry $Sol^3 \times \R$. 

If the polynomial $P_{m,n}$ has two equal roots, then we obtain the model space for the $Sol_0^4$ geometry, which is a semi-direct product $\R^3
\rtimes \R$, where the action of $\R$ on $\R^3$ is given by
\[
t \mapsto 
\left(\begin{array}{ccc}
   e^{t} & 0 & 0 \\
   0 & e^{t} & 0 \\
   0 & 0 & e^{-2t} \\
\end{array} \right).
\]

The main result in~\cite{KL} is that aspherical manifolds (more generally, rationally essential manifolds) are not dominated by direct products if their fundamental group is not presentable by products. 

\begin{defn}\label{d:PP}
A group $G$ is called {\em not presentable by products}\index{group!presentable by products} if for every homomorphism $\varphi\colon G_1\times G_2\longrightarrow G$ with $[G:\mathrm{im}(\varphi)]<\infty$, one of the images $\varphi(G_i)$ is finite.
\end{defn}

Manifolds modeled on one of the geometries $Sol_{m\neq n}^4$ or $Sol_0^4$ fulfill the above property:

\begin{prop}\cite[Prop. 6.13]{NeoIIPP}\label{p:sol&sol}
The fundamental group of a $4$-manifold which is modeled on the geometry $Sol_{m\neq n}^4$ or the geometry $Sol_0^4$ is not presentable by products.
\end{prop}

The last solvable model space is an extension of $\R$ by the $3$-dimensional Heisenberg group
\[
 Nil^3 = 
\Biggl\{ \left( \begin{array}{ccc}
  1 & x & z \\
  0 & 1 & y \\
  0 & 0 & 1 \\
\end{array} \right) \biggl\vert
\ x,y,z \in \R \Biggl\}.
\]
Namely, the Lie group $Sol_1^4$ is defined as a semi-direct product $Nil^3 \rtimes \R$, where $\R$ acts on $Nil^3$ by
\[
t \mapsto 
\left(\begin{array}{ccc}
   1 & e^{-t}x & z \\
   0 & 1 & e^{t}y \\
   0 & 0 & 1 \\
\end{array} \right).
\]

Manifolds modeled on this geometry have the following property:

\begin{prop}\cite[Prop. 6.15]{NeoIIPP}\label{p:sol1}
A $Sol_1^4$ manifold $M$ is finitely covered by an $S^1$ bundle over a mapping torus of $T^2$ with hyperbolic monodromy (i.e., over a $Sol^3$ manifold).
\end{prop}

Every $4$-manifold that carries a solvable non-product geometry is a mapping torus:

\begin{thm}\cite[Sections 8.6 and 8.7]{Hillman}\label{t:mappingtorisolvable} \
 \begin{itemize}
  \item[\normalfont{(1)}] A manifold modeled on the $Sol_0^4$ or the $Sol_{m \neq n}^4$ geometry is a mapping torus of a self-homeomorphism of $T^3$.
  \item[\normalfont{(2)}] A manifold modeled on the $Nil^4$ or the $Sol_1^4$ geometry is a mapping torus of a self-homeomorphism of a $Nil^3$-manifold.
 \end{itemize}
\end{thm}

The remaining two aspherical models are  irreducible symmetric geometries, the real and the complex hyperbolic, denoted
by $\mathbb{H}^4$ and $\mathbb{H}^2(\mathbb{C})$ respectively. 

Finally, we will need the following: 

\begin{thm}\cite[Theorem 10.1]{Wall2}\cite[Prop. 1]{Kotschick:4-mfds}\label{t:Wall4Dgeometries}
 If $M$ and $N$ are homotopy equivalent $4$-manifolds modeled on geometries $\mathbb{X}^4$ and $\mathbb{Y}^4$ respectively, then
$\mathbb{X}^4$ and $\mathbb{Y}^4$ are the same.
\end{thm}

In particular, {\em an aspherical geometric $4$-manifold $M$ is finitely covered by an $\mathbb{X}^4$ manifold if and only if it carries the geometry $\mathbb{X}^4$}.

\subsection{Wang's ordering}

Suppose $M$ is an aspherical $3$-manifold which is not modeled on one of the six aspherical geometries $\mathbb{H}^3$, $Sol^3$,
$\widetilde{SL_2}$, $\mathbb{H}^2 \times \R$, $Nil^3$ or $\R^3$. Then there is a finite family of splitting tori so that $M$ can be cut into pieces, called JSJ pieces (named after Jaco-Shalen and Johannson). $M$ is called a {\em non-trivial graph} manifold if all the JSJ pieces are Seifert. If there is a non-Seifert JSJ piece, then this piece must be hyperbolic by Perelman's proof of Thurston's geometrisation conjecture. In that case, $M$ is called a {\em non-graph} manifold. 

 Wang~\cite{Wangorder} ordered all aspherical $3$-manifolds and Kotschick and I~\cite{KN} extended this to include all rationally inessential $3$-manifolds:

\begin{figure}[ht!]
     \[
\xymatrix{
& & & & \mathbb{H}^2 \times \R \ar[rrr] \ar[rrrd] & & & \R^3 \ar[d]^{(p\leq 1)}\\
\mathbb{H}^3 \ar[r] & \mathrm{(NGRAPH)} \ar[l] \ar[r] &
\mathrm{(GRAPH)} 
\ar[rru] \ar[rrd] \ar[rr] & &
Sol^3 \ar[rrr]^{(p\leq 1)} & & & \#_p(S^2 \times S^1)\\
& & & &\widetilde{SL_2} \ar[rrr] \ar[rrru] & & & Nil^3 \ar[u]_{(p\leq 1)}}
\]
\caption{\small{Ordering $3$-manifolds by maps of non-zero degree~\cite{Wangorder,KN}.}}
\label{f:order3-mfds}
\end{figure}
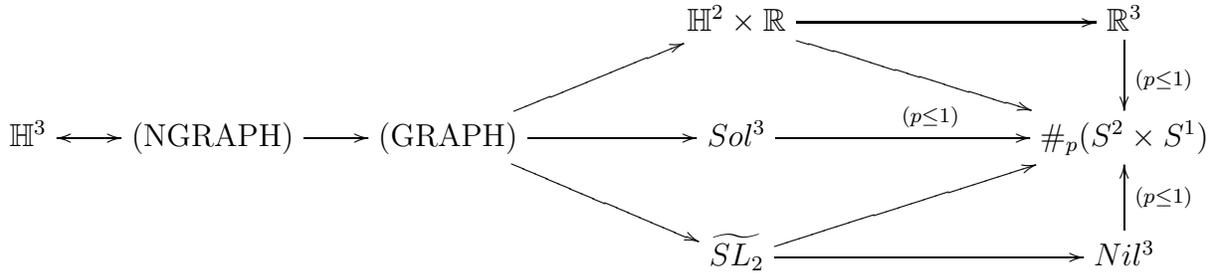
\begin{thm}[Wang's ordering]
\label{t:order3-mfds}
 Let the following classes of $3$-manifolds: 
\begin{itemize}
 \item[\normalfont{(i)}] aspherical\index{aspherical} and geometric, i.e., modeled on one of the six geometries $\mathbb{H}^3$, $Sol^3$, $\widetilde{SL_2}$, $\mathbb{H}^2 \times \R$,
$Nil^3$ or $\R^3$;
 \item[\normalfont{(ii)}] aspherical and non-geometric, i.e., $\mathrm{(GRAPH)}$ non-trivial graph or $\mathrm{(NGRAPH)}$ non-geometric irreducible non-graph;
 \item[\normalfont{(iii)}] rationally inessential, i.e., finitely covered by $\#_p(S^2 \times S^1)$, for some $p\geq 0$.
\end{itemize}
If there exists an oriented path from a class $X$ to another class $Y$ in Figure \ref{f:order3-mfds}, then any manifold in $Y$ is dominated by some manifolds in $X$. Otherwise, no manifold in $Y$ can be dominated by a manifold in $X$\index{Gromov!order}. 
\end{thm}

The proof of Theorem \ref{t:order3-mfds} for maps between (most) aspherical $3$-manifolds is given in~\cite{Wangorder} and for maps from $\mathbb{H}^2 \times \R$ manifolds to manifolds modeled on the geometries $\widetilde{SL_2}$ or $Nil^3$, or when the target manifold is finitely covered by $\#_p(S^2\times S^1)$, is given in~\cite{KN}. Note also some restrictions on the diagram concerning the number of summands in $\#_p(S^2\times S^1)$ for domination from  $Sol^3$, $Nil^3$ or $\R^3$ manifolds; see~\cite[pg. 4]{Neoorder}.

\subsection{Ordering the $4$-dimensional geometries}

Our goal in this section is to order in the sense of Wang all non-hyperbolic $4$-manifolds that carry a Thurston aspherical geometry:

\begin{thm}\label{t:order4}
 Consider all $4$-manifolds that are modeled on a non-hyperbolic aspherical geometry\index{manifold!aspherical}\index{manifold!geometric}\index{Thurston geometry}. If there is an oriented path from a geometry $\mathbb{X}^4$
to another geometry $\mathbb{Y}^4$ in Figure \ref{d:nonhypmaps}, then any $\mathbb{Y}^4$-manifold is dominated by an $\mathbb{X}^4$-manifold.
If there is no oriented path from $\mathbb{X}^4$ to $\mathbb{Y}^4$, then no $\mathbb{X}^4$-manifold dominates a $\mathbb{Y}^4$-manifold.\index{Gromov!order}
\end{thm}

\begin{figure}[ht!]
    \[
\xymatrix{
 &(\mathbb{H}^2 \times \mathbb{H}^2)_{irreducible}  & (\mathbb{H}^2 \times \mathbb{H}^2)_{reducible} \ar[rrd]& & & \\
& Sol_0^4 &   &                                       &  \mathbb{H}^2 \times \R^2 \ar[r] & \R^4\\
&Sol_{m\neq n}^4 & 
 \mathbb{H}^3 \times \R \ar[rr] \ar[rru] \ar[rrd] & & Sol^3
\times \R & &\\
&Sol_1^4 & & & \widetilde{SL_2} \times \R \ar[r] & Nil^3 \times \R\\
 & Nil^4  &  & & &}
\] 
\caption{\small{Ordering Thurston geometries in dimension four.}}
\label{d:nonhypmaps}
\end{figure}
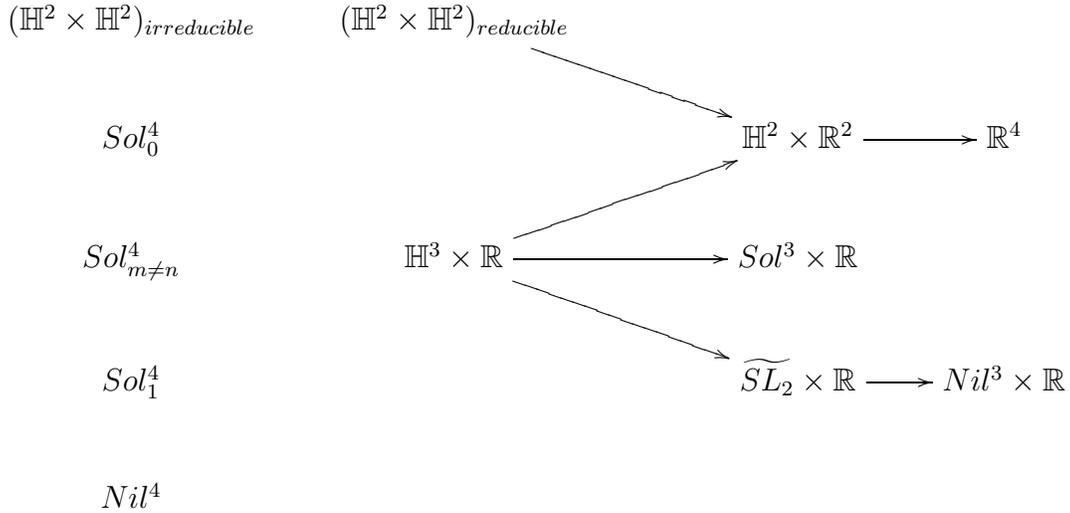 

Theorem \ref{t:order4} does not include the real or complex hyperbolic geometries, partially because some of the results about those geometries are well-known and because the domination relation for those geometries has been studied by other authors; see~\cite{CT,KL,Gaifullin}. Similarly, the non-aspherical geometries are not included in the above theorem; those geometries are either products or their representatives are simply connected, see~\cite{Neobranch, Neothesis} for further details.

We will devote the rest of this section in sketching a proof of Theorem \ref{t:order4}, and refer to~\cite{Neoorder} for the details.

\subsubsection{Manifolds covered by products}\label{s:products}

First, we will examine $4$-manifolds that are finitely covered by direct products. In other words, we will explain the right-hand side of Figure \ref{d:nonhypmaps}.

\subsubsection*{Non-existence stability between products}\index{domination}

Dealing with manifolds in dimension four, a natural question is whether one can extend Wang's ordering given by Theorem \ref{t:order3-mfds} to $4$-manifolds that are finitely covered by $N\times S^1$, where $N$ is a $3$-manifold as in Theorem \ref{t:order3-mfds}. The problem is whether the non-existence results by Wang extend in dimension four, namely, whether $M\ngeq N$ implies $M\times S^1 \ngeq N\times S^1$. This raises the following more general stability question:

\begin{prob}\label{pr:productstability}
 Suppose $M\ngeq N$. Does this imply $M\times W \ngeq N\times W$ for every manifold $W$?
 \end{prob}
 
  This problem is of independent interest, because, for example, our current knowledge on the multiplicativity of functorial numerical invariants (such as the simplicial volume) under taking products is not  enough to answer this kind of problems, even when an invariant remains non-zero under taking products; compare to \eqref{eq:simplicialproducts}.

The next result is based on the celebrated realisation theorem of Thom \cite{Thom} and gives a sufficient condition for non-domination stability for products:

\begin{thm}\cite{KotschickLoehNeofytidis,Neothesis}\label{t:mapsbetweenproducts}
 Let $M,N$ be $n$-manifolds such that $N$ is not dominated by products and $W$ be an $m$-manifold. Then, $M \geq N$ if and only if $M \times W \geq N \times W$.
\end{thm}

 In a similar vein, we have the following:

\begin{prop}\cite{KotschickLoehNeofytidis,Neothesis}\label{c:productslower}
 Let $M,W$ and $N$ be manifolds of dimensions $m,k$ and $n$ respectively such that $m,k<n<m+k$. If $N$ is not dominated by
products, then $M\times W\ngeq N\times V$, for any manifold $V$ of dimension $m+k-n$.
\end{prop}

\subsubsection*{Targets that are virtual products with a circle factor.}

Now we apply Theorem \ref{t:mapsbetweenproducts} to $4$-manifolds that are finitely covered by $N \times S^1$, thus extending  Theorem \ref{t:order3-mfds}. 
In the following theorem, we shall say that a $4$-manifold belongs to the class $X
\times \R$ if it is finitely covered by a product $N \times S^1$, where $N$ is a $3$-manifold that belongs to the class $X$ as defined in Theorem \ref{t:order3-mfds}.

\begin{thm}\label{t:4Dwangorder}
Let $X$ be one of the three classes (i)--(iii) given in Theorem \ref{t:order3-mfds}. If there exists an oriented path from a class $X$ to another class $Y$ in Figure \ref{f:order3-mfds}, then any $4$-manifold in $Y \times
\R$ is dominated by a manifold in $X \times \R$. Otherwise, no manifold in $Y \times \R$ can be dominated by a manifold in $X \times \R$.
\end{thm}

\begin{proof}
The existence part of Theorem \ref{t:4Dwangorder} follows easily by the corresponding existence results for maps between $3$-manifolds given in Theorem \ref{t:order3-mfds}, hence we concentrate on  the non-existence part. Note that there is no $4$-manifold in the class $(\#_p S^2 \times S^1) \times \R$ that can dominate a manifold in the other classes, since the latter are all rationally essential. Thus, the interesting cases are when both domain and target are aspherical.

We first deal with targets whose $3$-manifold factor $N$ in their finite cover $N \times S^1$ is not dominated by products. The proof of the following uses Proposition \ref{c:productslower} and Theorem \ref{t:mapsbetweenproducts}:

\begin{prop}\cite[Prop. 4.4]{Neoorder}\label{p:liftingtoproducts}
 Suppose $W$ and $Z$ are $4$-manifolds such that 
\begin{itemize}
 \item[\normalfont{(1)}] $W$ is dominated by products;
 \item[\normalfont{(2)}] $Z$ is finitely covered by $N \times S^1$, where $N$ is a $3$-manifold not dominated by products.
\end{itemize}
 If $W \geq Z$, then there exists a $3$-manifold $M$ such that $M \times S^1 \geq W$ and $M \geq N$. In particular, $M$ cannot be
dominated by products.
\end{prop}

By~\cite[Theorem 4]{KN}, only $\mathbb{H}^2 \times \R$ and $\R^3$ manifolds are dominated by products among the aspherical $3$-manifolds. Hence, Proposition \ref{p:liftingtoproducts} and the non-existence part of Theorem \ref{t:order3-mfds} imply the following:

\begin{cor}
If $Y \neq \mathbb{H}^2 \times \R,\R^3$, then the non-existence part of Theorem \ref{t:4Dwangorder} holds true for every aspherical target in $Y \times \R$.
 \end{cor}

In the Thurston geometric setting, we have the following straightforward consequence of Proposition \ref{p:liftingtoproducts}:

\begin{cor}
 Let $W$ and $Z$ be aspherical $4$-manifolds carrying product geometries $\mathbb{X}^3 \times \R$ and $\mathbb{Y}^3 \times \R$ respectively, such that $\mathbb{Y}^3\neq\mathbb{H}^2 \times \R,\R^3$. If $W \geq Z$, then every $\mathbb{Y}^3$ manifold is dominated by some $\mathbb{X}^3$ manifold.
\end{cor}

In order to finish the proof of Theorem \ref{t:4Dwangorder}, we need to show that manifolds modeled on $\mathbb{H}^2 \times \R^2$ or $\R^4$
are not dominated by $\widetilde{SL_2} \times \R$, $Sol^3 \times \R$ or $Nil^3 \times \R$ manifolds. For the latter two geometries, this follows by the growth of their Betti numbers (cf. Lemma \ref{l:Betti}).
We are left to deal with the $\widetilde{SL_2} \times \R$ geometry: Note that each $\R^4$ manifold is finitely covered by the 4-torus $T^4$ and, therefore, it is virtually dominated\footnote{$M$ {\em virtually dominates} $N$ if some finite cover of $M$ dominates $N$.} by every $\mathbb{H}^2 \times \R^2$ manifold. Thus, it suffices to show that $T^4$ cannot be dominated by a product $M \times S^1$, where $M$ is an  $\widetilde{SL_2}$ manifold. After passing to a finite cover, we can assume that $M$ is a non-trivial $S^1$ bundle over a hyperbolic surface $\Sigma$; see Table~\ref{table:3geom}. 
Suppose $f \colon M \times S^1 \longrightarrow T^4$ is a continuous map. The product $M \times S^1$ carries the structure of a non-trivial $S^1$ bundle over $\Sigma \times S^1$, by multiplying by $S^1$ both the total space $M$ and the base surface $\Sigma$ of the $S^1$ bundle $M \longrightarrow \Sigma$. The $S^1$ fiber of the circle bundle 
\[
S^1\longrightarrow M \times S^1\longrightarrow \Sigma \times S^1
\]
has finite order in $H_1(M\times S^1)$, being also the fiber of $M$. Therefore, its image under $H_1(f)$ has finite order in $H_1(T^4)$. Now, since $H_1(T^4)$ is isomorphic to $\pi_1(T^4) \cong \Z^4$, we deduce that $\pi_1(f)$ maps the fiber of the $S^1$ bundle $M \times S^1 \longrightarrow \Sigma \times S^1$ to the trivial element in $\pi_1(T^4)$. The latter implies that $f$ factors through the base $\Sigma \times S^1$, because the total space $M \times S^1$, the base $\Sigma \times S^1$ and the target $T^4$ are all aspherical. This implies that the degree of $f$ must be zero, completing the proof of Theorem \ref{t:4Dwangorder}.
\end{proof}

\subsubsection*{Manifolds covered by the product of two hyperbolic surfaces.}

We close this subsection by examining manifolds that are finitely covered by a product of two hyperbolic surfaces, i.e. reducible $\mathbb{H}^2 \times \mathbb{H}^2$ manifolds.

Clearly, every $4$-manifold modeled on $\mathbb{H}^2 \times \R^2$ or $\R^4$ is dominated by a product of two hyperbolic
surfaces. 
However, Proposition \ref{c:productslower} (or Proposition \ref{p:liftingtoproducts}) tells us that aspherical $4$-manifolds that are finitely covered by products $N \times S^1$, where $N$ does not belong to one of the classes $\mathbb{H}^2 \times \R$ or $\R^3$, cannot be dominated by products of hyperbolic surfaces. 

Finally, we need to show that there is no manifold modeled on an aspherical geometry $\mathbb{X}^3 \times \R$ which can dominate a product of two hyperbolic surfaces. The fundamental group of a product $M \times S^1$ has center at least infinite cyclic (coming from the $S^1$ factor), while the center of the fundamental group of a product of two hyperbolic surfaces $\Sigma_g \times \Sigma_h$ is trivial. Therefore, every map $f \colon M \times S^1 \longrightarrow \Sigma_g \times \Sigma_h$ (which we can assume $\pi_1$-surjective after passing to finite covers) kills the homotopy class of the $S^1$ factor of $ M \times S^1$, and so it factors through an aspherical manifold of dimension at most three, because both $M \times S^1$ and $\Sigma_g \times \Sigma_h$ are aspherical. Thus 
\[
H_4(f)([M \times S^1]) = 0 \in H_4(\Sigma_g \times \Sigma_h),
\]
meaning that $\deg(f)=0$.

\begin{rem}\label{r:othertech}
Note that the non-domination $M \times S^1 \ngeq \Sigma_g \times \Sigma_h$  (where $g,h\geq2$) follows also quickly by the fact that $M \times S^1$ has vanishing simplicial volume, whereas the simplicial volume of $\Sigma_g \times \Sigma_h$ is positive (by the inequalities in~\eqref{eq:simplicialproducts} or more generally by~\cite{Buch}).
However, we have chosen to give more elementary and uniform arguments for the proof of Theorem \ref{t:order4}, revealing also the strength of algebraic considerations alone.
\end{rem}

\subsubsection{Finishing the proof of Theorem \ref{t:order4}}\label{s:proofend}

Thus far, we have given a proof for the right-hand side of the diagram in Figure \ref{d:nonhypmaps}, i.e., concerning maps between geometric aspherical $4$-manifolds that are finitely covered by products. For the remaining parts in Figure \ref{d:nonhypmaps}, we need to show that each of the geometries $Nil^4$, $Sol_0^4$, $Sol_{m \neq n}^4$, $Sol_1^4$ and the irreducible geometry $\mathbb{H}^2 \times \mathbb{H}^2$ is not comparable with any other
(non-hyperbolic) geometry under the domination relation.

\subsubsection*{Comparing non-product solvable geometries}

We begin by showing that there are no maps of non-zero degree between any two manifolds that are modeled on different geometries among $Nil^4$, $Sol_0^4$,
$Sol_{m \neq n}^4$ or $Sol_1^4$. First, we deal with $Nil^4$ and $Sol^4_1$:

\begin{prop}
There are no maps of non-zero degree between $Nil^4$ and $Sol_1^4$ manifolds.
\end{prop}
\begin{proof}
$Nil^4$ manifolds and $Sol_1^4$ manifolds are finitely covered by $S^1$ bundles over $Nil^3$ manifolds and $Sol^3$ manifolds respectively, and the center of their fundamental groups remains infinite cyclic in finite covers; see Propositions \ref{p:nil4} and \ref{p:sol1} respectively. By Theorem \ref{t:order3-mfds}, there are no maps of non-zero degree between $Sol^3$ manifolds and $Nil^3$ manifolds, thus the proposition follows by the next lemma.
\end{proof}

\begin{lem}\cite[Lemma 5.1]{Neoorder}\label{l:factorizationbundles}
 Let $M_i \stackrel{p_i}\longrightarrow B_i$ ($i=1,2$) be $S^1$ bundles over aspherical manifolds $B_i$ of the same dimension, so that the
center of each $\pi_1(M_i)$ remains infinite cyclic in finite covers. If $B_1 \ngeq B_2$, then $M_1 \ngeq M_2$. 
\end{lem}

Next, we show that there are no maps of non-zero degree between $Sol_0^4$ manifolds and $Sol_{m \neq n}^4$ manifolds. Recall by Theorem \ref{t:mappingtorisolvable}(1) that any manifold modeled on any of these geometries is a mapping torus of $T^3$, and, moreover, the eigenvalues of the automorphism
of $\Z^3$ induced by the monodromy of $T^3$ are not roots of unity; cf.~\cite[pg. 164--165]{Hillman}.  The following general result in all dimensions shows that every non-zero degree map between such mapping tori is $\pi_1$-injective: 

\begin{prop}\cite[Prop. 5.3]{Neoorder}\label{p:injectivemappingtoriTn}
 Suppose $M$ and $N$ are finitely covered by mapping tori of self-homeomorphisms of $T^n$ so that the eigenvalues of the induced
automorphisms of $\Z^n$ are not roots of unity. If $f \colon M \longrightarrow N$ is a non-zero degree map, then $f$ is $\pi_1$-injective.
\end{prop}

Hence, by Theorem \ref{t:Wall4Dgeometries}, we deduce the following:

\begin{cor}\label{c:nonmapsbetweensol}
 Any two manifolds $M$ and $N$ modeled on $Sol_{m \neq n}^4$ and $Sol_0^4$ respectively are not comparable under the domination relation.
\end{cor}

Finally, in a similar vein, using Theorem \ref{t:mappingtorisolvable}, as well as Propositions \ref{p:nil4}, \ref{p:sol&sol} and \ref{p:sol1}, we obtain the following:

\begin{prop}\cite[Prop. 5.5 and 5.6]{Neoorder}\label{p:comparesolnil}
 If $M$ is a $Nil^4$ or $Sol_1^4$ manifold and $N$ is a $Sol_{m \neq n}^4$ or $Sol_0^4$ manifold, then there is no map of non-zero degree between $M$ and $N$.
\end{prop}

\subsubsection*{Non-product solvable manifolds vs virtual products}

We now indicate why there are no maps of non-zero degree between a $Nil^4$, $Sol_0^4$, $Sol_{m \neq n}^4$
or $Sol_1^4$ manifold and a manifold modeled on $\mathbb{X}^3 \times \R$ or on the reducible $\mathbb{H}^2 \times \mathbb{H}^2$ geometry. We need the following result, parts of which use the property {\em group not infinite-index presentable by products}, which will be defined and discussed briefly in Section~\ref{monotonicityKodaira}:

\begin{thm}\cite[Theorem F]{NeoIIPP}\label{t:productssolvable}
An aspherical geometric $4$-manifold $M$ is dominated by a non-trivial product if and only if it is finitely covered by a product. Equivalently, $M$ is modeled on one of the product geometries $\mathbb{X}^3 \times\R$ or the reducible $\mathbb{H}^2 \times \mathbb{H}^2$ geometry.
\end{thm}

In particular, we have:

\begin{cor}
A $4$-manifold modeled on one of the geometries $Nil^4$, $Sol_0^4$, $Sol_{m \neq n}^4$ or $Sol_1^4$ is not dominated by products.
\end{cor}

The proof of the converse uses again the structure theorems for the geometries $Nil^4$, $Sol_0^4$, $Sol_{m \neq n}^4$ and $Sol_1^4$ (see Section \ref{ss:Thurstonclassification}), as well as the growth of their Betti numbers (see for example \cite[Sections 8.6 and 8.7]{Hillman} and \cite[Section 6]{NeoIIPP}):

\begin{prop}\cite[Prop. 5.8 and 5.9]{Neoorder}\label{p:nilsolvsproducts}
 A manifold modeled on one of the geometries $Nil^4$, $Sol_0^4$, $Sol_{m \neq n}^4$ or $Sol_1^4$ does not dominate any manifold modeled on a geometry $\mathbb{X}^3 \times \R$ or the reducible $\mathbb{H}^2\times \mathbb{H}^2$ geometry.
\end{prop}
 
\subsubsection*{The irreducible $\mathbb{H}^2 \times \mathbb{H}^2$ geometry}\label{ss:irreducibleH2xH2}

Finally, we deal with the irreducible $\mathbb{H}^2 \times \mathbb{H}^2$ geometry. 

\begin{prop}\label{p:irreducible}
An irreducible $\mathbb{H}^2 \times \mathbb{H}^2$ manifold $M$ is not comparable under the domination relation with any other manifold possessing a non-hyperbolic aspherical geometry.
\end{prop}
\begin{proof}
Suppose first that $f \colon M \to N$ is a map of non-zero
degree, where $N$ is an aspherical manifold which is not modeled on the irreducible $\mathbb{H}^2 \times \mathbb{H}^2$ geometry. After possibly passing to a finite cover, we may assume that $f$ is $\pi_1$-surjective; in particular, we have the following short exact sequence
\[
 1 \longrightarrow \ker (\pi_1(f)) \longrightarrow \pi_1(M) \stackrel{\pi_1(f)}\longrightarrow \pi_1(N) \longrightarrow 1.
\]
By \cite[Theorem IX.6.14]{Margulis}, the kernel $\ker (\pi_1(f))$ is trivial, and thus $\pi_1(f)$ is an isomorphism. Since $M$ and $N$ are
aspherical, we conclude that $M$ is homotopy equivalent to $N$, which contradicts Theorem \ref{t:Wall4Dgeometries}. Hence, $M \ngeq N$.

Conversely, we claim that $M$ is not dominated by any non-hyperbolic geometric aspherical $4$-manifold $N$. Since $M$ is not dominated by
products (e.g., by Theorem \ref{t:productssolvable}), it suffices to show that $N\ngeq M$ when $N$ is modeled on one of the geometries $Sol_1^4$, $Nil^4$, $Sol_{m \neq n}^4$
or
$Sol_0^4$. For any of those four geometries, $\pi_1(N)$ has a normal
subgroup of infinite index, which is free Abelian of rank one (for the geometries $Sol_1^4$ and $Nil^4$) or three (for the geometries $Sol_{m \neq n}^4$ and $Sol_0^4$);
see Section \ref{ss:Thurstonclassification} and~\cite[Section 6]{NeoIIPP}.
If $f \colon N \to M$ is a ($\pi_1$-surjective) map of non-zero degree, then by \cite[Theorem IX.6.14]{Margulis} either $f$ factors through a
lower dimensional aspherical manifold or $\pi_1(M)$ is free Abelian of finite rank. None of these cases can occur.
\end{proof}

We have now completed the proof of Theorem \ref{t:order4}.

\section{Geometric Kodaira dimension, monotonicity, and simplicial volume}

The Kodaira dimension is an important tool in the classification of complex manifolds, and has been generalised to various classes, such as symplectic manifolds and almost complex manifolds; we refer to~\cite{Lisurvey} for a survey. Following Thurston's geometrisation picture, we will introduce an axiomatic definition for the Kodaira dimension and show that this geometric Kodaira dimension is monotone with respect to the domination relation for manifolds of dimension $\leq5$. We will also compare the geometric Kodaira dimension with other, existing notions of Kodaira dimension, and  establish a relationship to the simplicial volume. 

\subsection{Kodaira dimension}

A substantial attempt to introduce a notion of Kodaira dimension\index{Kodaira dimension}\index{dimension!Kodaira} for non-complex manifolds, in particular for odd-dimensional manifolds, was made by W. Zhang~\cite{Zhang}. Recall that Kodaira's original approach defines the holomorphic Kodaira dimension $\kappa^h(M, J)$\index{Kodaira dimension!holomorphic} for a complex manifold $(M,J)$ of real dimension $2m$ by 
\begin{equation}\label{eq:Kodairah}
\kappa^h(M,J)= \left\{\begin{array}{ll}
        -\infty, & \text{if } P_l(M,J)=0 \text{ for all } l\geq1;\\
        \ \ 0, & \text{if } P_l(M,J)\in\{0,1\}, \text{ but } \not\equiv0 \text{ for all } l\geq1;\\
        \ \ k, & \text{if } P_l(M,J)\sim cl^k, c > 0.
        \end{array}\right.
\end{equation}
where $P_l(M, J)$ denotes the $l$-th plurigenus of the complex manifold $(M, J)$ defined by $P_l(M,J)=h^0(K_J{\otimes}^l)$, with $K_J$ the canonical bundle of $(M, J)$. Zhang introduced the notion of {\em geometric} (or {\em topological}) {\em Kodaira dimension} for $3$-manifolds and geometric $4$-manifolds, following the principle suggested by \eqref{eq:Kodairah} that compact geometries have the smallest value ($-\infty$), while hyperbolic geometries have the biggest value (half of the dimension of the manifold). Subsequently, Zhang and I~\cite{NZ} introduced a more unified approach which we present below. As we shall see, this unification includes as well many non-geometric situations.

\subsubsection{Axiomatic definition of $\kappa^g$}\label{ss:ad}

Let $\mathcal G$ be the smallest class of manifolds which contains all of the following:
\begin{itemize}
\item points;
\item manifolds that carry a compact geometry;
\item solvable manifolds (solvable-by-solvable);
\item irreducible symmetric spaces of non-compact type;
\item fibrations or manifolds that carry a fibered geometry, so that their fiber and base (geometries) belong in $\mathcal G$.
\end{itemize}

\begin{defn}\label{d:gK}
Let $M$ be an $n$-manifold in $\mathcal G$. We define its {\em (geometric) Kodaira dimension}\index{manifold!geometric}
\[
\kappa^g(M)\in\biggl\{-\infty,0,1,\frac{3}{2},2,...,\frac{n}{2}\biggl\}
\]
by the following axioms:
\begin{itemize}
\item[(A0)] If $M$ is a point, then we set $\kappa^g(M)=0$;
\item[(A1)] If $M$ carries a compact geometry, then $\kappa^g(M)=-\infty$;
\item[(A2)] If $M$ is of solvable type,
then $\kappa^g(M)=0$;
\item[(A3)] If $M$ is irreducible symmetric of non-compact type, then $\kappa^g(M)=\frac{n}{2}$;
\item[(A4)] If $M$ is a fiber bundle or carries a fibered geometry $\mathbb F\to\mathbb X^n\to\mathbb B$, such that it does not satisfy any of (A1)-(A3), then
$$
\kappa^g(M)=\sup_{F,B}\{\kappa^g(F)+\kappa^g(B)\},
$$
where the supremum runs over all manifolds $F$ and $B$ which can occur in a fibration $F\to M\to B$ or are modeled on a geometry $\mathbb F$ and $\mathbb B$ respectively, and which  satisfy one of (A1)-(A3).
\end{itemize}
\end{defn}

Note  the following consequence of the above axioms:

\begin{lem}\label{l:cover}
Let  $M\in\mathcal G$. If $\overline M\to M$ is a finite cover, then $\overline M\in\mathcal G$ and $\kappa^g(\overline M)=\kappa^g(M)$.
 \end{lem}
 
\subsubsection{Classification in dimensions $\leq5$}
We will now classify manifolds up to dimension five according to their Kodaira dimension.

\subsection*{0-manifolds}
By (A0), the Kodaira dimension of a point is zero.

\subsection*{1-manifolds}
The circle $S^1=\R/\Z$ is modeled on the real line, hence $\kappa^g(S^1)=0$ by (A2).

\subsection*{2-manifolds}\index{Thurston geometry}
Let $\Sigma_h$ be a surface of genus $h$. First, if $h=0$, then $\Sigma_0=S^2$, hence $\kappa^g(\Sigma_0)=0$ by (A1). Next, if $h=1$, then $\Sigma_1=T^2=\R^2/\Z^2$ and thus $\kappa^g(\Sigma_1)=0 $ by (A2). Lastly, if $h\geq2$, then $\Sigma_h$ is hyperbolic, and so $\kappa^g(\Sigma_h)=2/2=1$ by (A3). The above are summarised in Table \ref{eq.dim2}.
\begin{table}[!ht]
\centering
\begin{tabular}{c|c}
$\kappa^g$ & Geometry\\
\hline
         $-\infty$   & $S^2$\\
         $0$    & $\R^2$ \\
$1$ & $\mathbb{H}^2$\\            
\end{tabular}
\vspace{9pt}
\caption{{\small The Kodaira dimension for surfaces}}\label{eq.dim2}
\end{table}

\subsection*{3-manifolds}
The geometry $S^3$ satisfies (A1), the geometries $\R^3$, $Nil^3$ and $Sol^3$ satisfy (A2), and the geometry $\mathbb{H}^3$ satisfies (A3). We are left with three geometries which do not satisfy any of (A1)-(A3). For  $S^2\times\R$, Axiom (A4) and the Kodaira dimensions for 1- and 2-manifolds imply
\[
\kappa^g(S^2\times S^1)=\kappa^g(S^2)+\kappa^g(S^1)=-\infty.
\]
Finally, since any 3-manifold $M$ modeled on the $\mathbb{H}^2\times\R$ or $\widetilde{SL_2}$ geometry is finitely covered by an $S^1$ bundle over a hyperbolic surface $\Sigma_h$, Axiom (A4), Lemma \ref{l:cover} and the Kodaira dimensions for $S^1$ and hyperbolic surfaces imply 
\[
\kappa^g(M)=\kappa^g(S^1)+\kappa^g(\Sigma_h)=1.
\]
Table \ref{eq.dim3} summarises the above values of $\kappa^g$.
\begin{table}[!ht]
\begin{tabular}{c|c}
$\kappa^g$ & Geometry\\
\hline
         $-\infty$   & $S^3$, $S^2\times\R$\\
         $0$    & $\R^3$, $Nil^3$, $Sol^3$ \\
$1$ & $\mathbb{H}^2\times\R$, $ \widetilde{SL_2}$\\     
$\frac{3}{2}$ & $\mathbb{H}^3$\\          
\end{tabular}
\vspace{9pt}
\caption{{\small The Kodaira dimension for geometric 3-manifolds}}\label{eq.dim3}
\end{table}

\subsection*{4-manifolds}
If a manifold $M$ is modeled on one of the three compact geometries $S^4$, $\CP^2$ and $S^2\times S^2$, then $\kappa^g(M)=-\infty$ by (A1). The geometries $\R^4$, $Nil^4$, $Nil^3\times\R$, $Sol_0^4$, $Sol_1^4$ and $Sol_{m,n}^4$ satisfy (A2), hence  $\kappa^g(M)=0$ for any manifold $M$ modeled on any of these geometries. If $M$ is modeled on one among $\mathbb{H}^4$, $\mathbb{H}^2(\mathbb C)$ or the irreducible $\mathbb{H}^2\times\mathbb{H}^2$ geometry, then it satisfies (A3), hence $\kappa^g(M)=4/2=2$. We are left with seven geometries which fall in Axiom (A4): If $M$ carries one among the geometries $S^2\times\R^2$, $S^2\times\mathbb{H}^2$ or $S^3\times\R$, then it is finitely covered by a manifold which is a fiber bundle with fiber one of the compact manifolds $S^2$ or $S^3$. Thus $\kappa^g(M)=-\infty$, because $\kappa^g(S^n)=-\infty$ for $n\geq2$. If $M$ is modeled on one of the geometries $\mathbb{H}^2\times\R^2$ or $\widetilde{SL_2}\times\R$, then $\kappa^g(M)=1$ by the corresponding classifications in lower dimensions. Similarly, if $M$ is modeled on $\mathbb{H}^3\times\R$, then $\kappa^g(M)=3/2$. Finally, if $M$ carries the reducible geometry $\mathbb{H}^2\times\mathbb{H}^2$, then it is finitely covered by a product of two hyperbolic surfaces, hence $\kappa^g(M)=2$, since  hyperbolic surfaces have Kodaira dimension one. We summarise these values in Table \ref{eq.dim4}. 

\begin{table}[!ht]
\begin{tabular}{c|c}
$\kappa^g$ & Geometry\\
\hline
         $-\infty$   & $S^4$, $\CP^2$, $S^2\times \mathbb{X}^2$, $S^3\times\R$\\
         $0$    & $\R^4$, $Nil^4$, $Nil^3\times\R$, $Sol_{m,n}^4$, $Sol_0^4$, $Sol_1^4$\\
$1$ & $\mathbb{H}^2\times\R^2$, $\widetilde{SL_2}\times\R$\\     
$\frac{3}{2}$ & $\mathbb{H}^3\times\R$\\     
$2$ &  $\mathbb{H}^4$,  $\mathbb{H}^2(\mathbb C)$, $\mathbb{H}^2\times\mathbb{H}^2$\\      
\end{tabular}
\vspace{9pt}
\caption{{\small The Kodaira dimension for geometric 4-manifolds}}\label{eq.dim4}
\end{table}

\begin{rem}\label{r:variousfibrations}
The geometry $\mathbb H^3\times\R$ is one of the first examples which indicate our new approach to introduce systematically half-integer values for $\kappa^g$, distinguishing thus further the various classes of manifolds by their Kodaira dimension. In addition, this example reveals the usefulness and necessity of (A4): A 4-manifold $M$ modeled on the geometry $\mathbb H^3\times\R$ is finitely covered by $F\times S^1$ for some hyperbolic $3$-manifold $F$. In addition, $F$ is (up to finite covers) a mapping torus of a pseudo-Anosov diffeomorphism of a hyperbolic surface $\Sigma$ (see Table~\ref{table:3geom}), hence, in particular, $M$ is (covered by) a fiber bundle $\Sigma\to M\to T^2$. Thus, we compute by the values of $\kappa^g$ in dimensions $\leq3$ that
\[
\kappa^g(M)=\sup\{\kappa^g(F)+\kappa^g(S^1),\kappa^g(\Sigma)+\kappa^g(T^2)\}=\biggl\{\frac{3}{2},1\biggl\} =\frac{3}{2}.
\]
\end{rem}

\begin{rem}\label{comparingKodairasol}
The values of the geometric Kodaira dimension of $4$-manifolds match with the values of the holomorphic Kodaira dimension for K\"ahler manifolds. However, according to (A2), the Kodaira dimension for $Sol_0^4$ and $Sol_1^4$ manifolds is zero instead of $-\infty$ as defined in~\cite{Zhang}, 
following Wall's scheme for complex non-K\"ahler surfaces~\cite{Wall2}. 
We could have imposed further conditions (e.g., on the virtual second Betti number) so that our Kodaira dimension for  those manifolds is $-\infty$, 
however, we have chosen to keep our axiomatic approach natural with minimal assumptions. Indeed, the value $\kappa^g=0$ here not only follows by (A2), but it is also compatible with (A4).
\end{rem}

\subsection*{5-manifolds}
The 5-dimensional Thurston geometries were classified by Geng~\cite{Geng1}. According to Geng's list, there exist fifty eight geometries, of which fifty four are realised by compact manifolds. We will only enumerate the latter geometries according to Definition \ref{d:gK}, and refer the reader to the three papers from Geng's thesis~\cite{Geng1,Geng2,Geng3}, as well as to the references in~\cite{Geng1}, for further details. Finally, as it is remarked in~\cite[Section 4]{Geng1}, a similar classification for the Thurston geometries was partially done in dimensions six and seven. In particular, one can use Definition \ref{d:gK} to determine the Kodaira dimensions of those manifolds.

\subsubsection*{(A1).}
There are three geometries of compact type, namely, the 5-sphere $S^5$, the Wu symmetric manifold $SU(3)/SO(3)$, and $S^2\times S^3$. If a manifold $M$ carries any of these geometries, then 
$\kappa^g(M)=-\infty$.

\subsubsection*{(A2).}
There are twenty geometries of solvable type. First, there exist two nilpotent and six solvable but not nilpotent geometries of type $\R^4\rtimes\R$, denoted by
$A_{5,1}, A_{5,2}$  and $A_{5,7}^{a,b,-1-a-b}$, $A_{5,7}^{1,-1-a,-1+a}, A_{5,7}^{1,-1,-1}, A_{5,8}^{-1}, A_{5,9}^{-1,-1}, A_{5,15}^{-1}$
respectively. Next, there are two nilpotent semi-direct products $Nil^4\rtimes\R$, denoted by $A_{5,5}$ and $A_{5,6}$. Furthermore, there is one nilpotent and one solvable but not nilpotent geometry of type $(\R\times Nil^3)\rtimes\R$, which are denoted by $A_{5,3}$ and $A_{5,20}^0$ respectively. Also, there is a solvable but not nilpotent geometry $\R^3\rtimes\R^2$ denoted by $A_{5,33}^{-1,-1}$. The last irreducible solvable-type geometry is $Nil^5$. The rest of those geometries are products of lower dimensional geometries, namely $\R^5$, $Nil^3\times\R^2$, $Nil^4\times\R$, $Sol_0^4\times\R$, $Sol_1^4 \times\R$, and $Sol_{m,n}^4\times\R$ (here $Sol_{m,m}^4\times\R=Sol^3\times\R^2$). If a manifold $M$ is modeled on any of the above geometries, then $\kappa^g(M)=0$ by (A2).

\subsubsection*{(A3).}
If a manifold $M$ carries one of the irreducible symmetric geometries of non-compact type $\mathbb{H}^5$ or $SL(3,\R)/SO(3)$, then $\kappa^g(M)=\frac{5}{2}$.

\subsubsection*{(A4).}
For the remaining geometries, we obtain a variety of values. First,
suppose $M$ is a manifold which is modeled on one among the following sixteen geometries:
         $S^2\times S^2\times\R$, $S^2\times\R^3, S^2\times Nil^3, S^2\times Sol^3,
         S^2\times\mathbb H^2\times\R, S^2\times\widetilde{SL_2},  S^2\times\mathbb H^3, S^2\times\mathbb H^3,
         S^3\times\R^2, S^3\times\mathbb H^2$, $S^4\times\R,  \CP^2\times\R,
          Nil^3\times_\R S^3, \widetilde{SL_2}\times_\alpha S^3, L(a,1)\times_{S^1}L(b,1)$, or $T^1(\mathbb H^3)$. Then $M$ has a fiber or base which is one of the compact geometries $S^2$, $S^3$, $S^4$ or $\CP^2$. Thus, $\kappa^g(M)=-\infty$ by the classification of Kodaira dimensions of manifolds of dimension $\leq4$.

Next, suppose $M$ is modeled on one of the geometries 
         $\R^3\times\mathbb H^2, Nil^3\times\mathbb H^2, Sol^3\times\mathbb H^2$, 
        $\widetilde{SL_2}\times\R^2, \R^2\rtimes\widetilde{SL_2}$, or $Nil^3\times_\R\widetilde{SL_2}$. Then $M$ fits into a fibration,  where the involved geometries are $\mathbb H^2$ and some solvable-type geometry. Therefore, $\kappa^g(M)=1$.

If $M$ carries the geometry $\mathbb H^3\times\R^2$, then $\kappa^g(M)=\frac{3}{2}$, where the supremum is achieved by the geometries $\mathbb H^3$ and $\R^2$ (compare to Remark \ref{r:variousfibrations}).

Next, suppose that $M$ carries one of the geometries 
         $\mathbb H^2\times\widetilde{SL_2}, \mathbb H^2\times\mathbb H^2\times\R, \widetilde{SL_2}\times_\alpha\widetilde{SL_2}$,
        $\mathbb H^4\times\R, \mathbb H^2(\mathbb C)\times\R$, or $\widetilde{U(2,1)/U(2)}$.
Each of these geometries fibers over one of the $4$-dimensional geometries $\mathbb H^2\times\mathbb H^2$, $\mathbb H^2$ or $\mathbb H^2(\mathbb C)$, which have Kodaira dimension two; see Table \ref{eq.dim4}.
We conclude that $\kappa^g(M)=2$.

Finally, if a manifold $M$ carries the geometry $\mathbb H^2\times\mathbb H^3$, then $\kappa^g(M)=1+\frac{3}{2}=\frac{5}{2}$.

We summarise the above in Table \ref{eq.dim5}.

\begin{table}[!ht]
\begin{tabular}{c|c}
$\kappa^g$ & Geometry\\
\hline
         $-\infty$   & $SU(3)/SO(3)$, $S^5$, $S^2\times \mathbb{X}^3$, $S^3\times \mathbb{X}^2$, $S^4\times\R$, $\CP^2\times\R$,\\
                         & $Nil^3\times_\R S^3$, $\widetilde{SL_2}\times_\alpha S^3$, $L(a,1)\times_{S^1}L(b,1)$, $T^1(\mathbb H^3)$\\
         $0$    & $\R^5$, $\R^4\rtimes\R$, $\R^3\rtimes\R^2$, $Nil^5$, $Nil^4\rtimes\R$, $(\R\times Nil^3)\rtimes\R$,\\ 
                   & $Nil^4\times\R$, $Nil^3\times\R^2$, $Sol_0^4\times\R$, $Sol_1^4\times\R$, $Sol_{m,n}^4\times\R$\\
       $1$      & $\mathbb{H}^2\times\R^3$, $\mathbb{H}^2\times\Nil^3$, $\mathbb{H}^2\times Sol^3$, $\R^2\times\widetilde{SL_2}$,  $\R^2\rtimes\widetilde{SL_2}$, $Nil^3\times_\R\widetilde{SL_2}$\\
$\frac{3}{2}$ & $\mathbb{H}^3\times\R^2$\\     
$2$ &  $\mathbb{H}^2\times\widetilde{SL_2}$, $\widetilde{SL_2}\times_\alpha\widetilde{SL_2}$, $\mathbb{H}^2\times \mathbb{H}^2\times\R$, $\mathbb{H}^4\times\R$, $\mathbb{H}^2(\mathbb C)\times\R$, $\widetilde{U(2,1)/U(2)}$\\      
$\frac{5}{2}$ & $\mathbb{H}^5$, $SL(3,\R)/SO(3)$, $\mathbb H^3\times\mathbb H^2$.\\  
\end{tabular}
\vspace{9pt}
\caption{{\small The Kodaira dimension for geometric 5-manifolds}}\label{eq.dim5}
\end{table}

\subsection{Monotonicity of the Kodaira dimension}\label{monotonicityKodaira}

One of the main motivations in~\cite{Zhang} was to study whether the geometric Kodaira dimension is monotone\index{invariant!monotone} with respect to the domination relation\index{domination} (see Definition \ref{d:monotone}, where $-\infty$ is allowed as well). To this end, Zhang defines the Kodaira dimension for all $3$-manifolds according to Thurston's picture: Let $M$ be a $3$-manifold. If it carries a Thurston geometry, then its Kodaira dimension is given by Table \ref{eq.dim3}. We call each of the values $-\infty,0,1,\frac{3}{2}$ the category to which a geometric $3$-manifold belongs. If $M$ does not carry any of the eight Thurston geometries, then  consider first its Kneser-Milnor prime decomposition (which is trivial when $M$ is prime) and then a toroidal decomposition for each prime summand of $M$, so that each piece of the toroidal decomposition carries one of the eight geometries with finite volume. We call this a $T$-decomposition of $M$. The Kodaira dimension of $M$ is then given in Table \ref{t:3dKodaira}.

\begin{table}[!ht]
\begin{tabular}{c|c}
$\kappa^g$ & For any $T$-decomposition of $M$...\\
\hline
         $-\infty$   &  each piece belongs to the category $-\infty$\\
         $0$  &   there is at least one piece in the category $0$, 
          but no pieces in the category $1$ or $\frac{3}{2}$ \\
$1$ &  there is at least one piece in the category $1$, 
 but no pieces in the category $\frac{3}{2}$\\     
$\frac{3}{2}$ &  there is at least one piece in the category $\frac{3}{2}$ (hyperbolic piece)\\          
\end{tabular}
\vspace{9pt}
\caption{{\small The Kodaira dimension for 3-manifolds}}\label{t:3dKodaira}
\end{table}

The main result of~\cite{Zhang} is the following:

\begin{thm}\label{t:monotonedim3}
Let $M,N$ be $3$-manifolds. If $M\geq N$, then $\kappa^g(M)\geq\kappa^g(N)$.
\end{thm}

\begin{rem}
Note that Theorem \ref{t:monotonedim3} is also a consequence of Theorem \ref{t:order3-mfds}, which is sharper in the sense that it tells us the existence or not of maps between manifolds modeled on  different geometries with the same Kodaira dimension, while Theorem \ref{t:monotonedim3} does not.
\end{rem}

Subsequently, Zhang asked whether the monotonicity result for the Kodaira dimension in Theorem \ref{t:monotonedim3} could be extended in higher dimensions. For geometric $4$-manifolds, this is a consequence of the ordering given in Theorem \ref{t:order4}:

\begin{thm}\label{t:monotonedim4}
Let $M,N$ be geometric $4$-manifolds. If $M\geq N$, then $\kappa^g(M)\geq\kappa^g(N)$.
\end{thm}

In my recent paper with Zhang~\cite{NZ}, we showed that $\kappa^g$ is monotone for geometric $5$-manifolds:

\begin{thm}\label{t:monotonedim5}
Let $M,N$ be geometric $5$-manifolds. If $M\geq N$, then $\kappa^g(M)\geq\kappa^g(N)$.
\end{thm}

We will only summarise the basic steps of the proof, pointing out some techniques and phenomena, and refer to~\cite{NZ} for the details.

\begin{proof}[Sketch of proof]
We need to show that if $\kappa^g(M)<\kappa^g(N)$, then $M\ngeq N$. Hence, we need to examine the various cases according to Table \ref{eq.dim5}. First, we observe that if $\kappa^g(N)=5/2$, then $\|N\|>0$, whereas $\|M\|=0$ whenever $\kappa^g(M)\neq 5/2$, and thus $M\ngeq N$; these use Gromov's results~\cite{Gromov}, the approximations given by \eqref{eq:simplicialproducts}, as well as a result by Bucher~\cite{Bucher} for the geometry $SL(3,\R)/SO(3)$. We are now left to examine the cases 
\[
\kappa^g(M)\in\{-\infty,0,1,3/2\} \ \text{and} \ \kappa^g(N)\neq\frac{5}{2}.
\] 

If $\kappa^g(M)=-\infty$, then $M$ is rationally inessential and thus it cannot dominate rationally essential manifolds. But if $\kappa^g(N)\neq-\infty$, then $N$ is aspherical, in particular rationally essential. Hence $M\ngeq N$.

If $\kappa^g(M)=0$, then $M$ is modeled on a solvable-type geometry, whereas if $\kappa^g(N)>0$, then $\pi_1(N)$ is not solvable. Hence, the non-domination $M\ngeq N$ follows by the fact that if $\varphi\colon H_1\longrightarrow H_2$ is a group homomorphism with $H_1$ solvable, then the image $\varphi(H_1)\subseteq H_2$ is solvable. 

Suppose now that $\kappa^g(M)=1$. This is the most delicate case and requires a step-by-step examination of many geometries individually.  Among the most interesting cases occur when $M$ is a $Sol^3\times\mathbb H^2$ manifold and $N$ is finitely covered by a non-trivial circle bundle over a hyperbolic or an $\mathbb H^2\times\mathbb H^2$ manifold, because this reveals some new group theoretic phenomena. Recall (cf. Definition \ref{d:PP}) that a group $G$ is called  presentable by products if there exist two infinite elementwise commuting subgroups $G_1,G_2\subseteq G$, so that the image of the multiplication homomorphism 
\[
G_1\times G_2\longrightarrow G
\]
 has finite index in $G$. If in addition both $G_i$ can be chosen with 
 \[
 [G:G_i]=\infty,
 \]
  then $G$ is called {\em infinite index presentable by products} or {\em IIPP}\index{group!IIPP}. This notion was defined in~\cite{NeoIIPP} and it is a sharp refinement between {\em reducible} groups\index{group!reducible} (that is, groups that are up to finite-index subgroups direct products of two infinite groups), and groups presentable by products, i.e.,
\[
\{\text{reducible groups}\}\subsetneq\{\text{IIPP groups}\}\subsetneq\{\text{groups presentable by products}\}.
\]
The following result gives a criterion for the equivalence between IIPP and reducible for central extensions:

\begin{thm}\cite[Theorem D]{NeoIIPP}\label{t:IIPPcriterion}
Let $G$ be a group with center $C(G)$ such that the quotient $G/C(G)$ is not presentable by products. Then, $G$ is reducible if and only if it is IIPP.
\end{thm}

Non-elementary hyperbolic groups is a standard prominent class of groups that are not presentable by products~\cite{KL}. If $N$ is modeled on the geometry $\widetilde{U(2,1)/U(2)}$, then $N$ is finitely covered by a non-trivial $S^1$ bundle over a complex hyperbolic 4-manifold $B$. Since $\pi_1(N)$ is not reducible,  Theorem  \ref{t:IIPPcriterion} implies that $\pi_1(N)$ is not IIPP. Hence, any map from a manifold modeled on the geometry $Sol^3\times\mathbb H^2$ to $N$ has degree zero, by the next theorem:

\begin{thm}\cite[Theorem B]{NeoIIPP}\label{t:IIPPdomination}
Let $N$ be an $S^1$ bundle over an aspherical manifold $B$, so that $\pi_1(N)$ is not IIPP and its center remains infinite cyclic in finite covers. Then $P\ngeq N$, for any non-trivial direct product $P$.
\end{thm}

We remark that the same argument applies when the target $N$ is a non-trivial $S^1$ bundle over a 4-manifold which is modeled on the irreducible $\mathbb H^2\times\mathbb H^2$ geometry. Note, however, that Theorem \ref{t:IIPPcriterion} does not hold anymore if we remove the  condition on the quotient group $G/C(G)$ being not presentable by products. For instance, the fundamental group of a $Nil^5$ manifold $N$ is irreducible and IIPP (see~\cite[Section 8]{NeoIIPP}), and fits into the following central extension
\[
1\longrightarrow\Z\longrightarrow\pi_1(N)\longrightarrow\Z^4\longrightarrow1.
\]
It was shown in~\cite[Section 8]{NeoIIPP} that $N$ does not admit maps of non-zero degree from non-trivial direct products. In a similar vein, one proves that $M\ngeq N$ when $N$ is modeled on the geometry $\widetilde{SL_2}\times_\alpha\widetilde{SL_2}$, since in that case $N$ is (finitely covered by) a non-trivial $S^1$ bundle over the product of two hyperbolic surfaces $\Sigma_{g}\times\Sigma_{h}$, its fundamental group fits into the central extension
\[
1\longrightarrow\Z\longrightarrow\pi_1(N)\longrightarrow\pi_1(\Sigma_{g})\times\pi_1(\Sigma_{h})\longrightarrow1,
\]
and it is moreover irreducible and IIPP.

Finally, if $\kappa^g(M)=\frac{3}{2}$, i.e., $M$ is modeled on $\mathbb H^3\times\R$, then $M$ is (finitely covered by) a product of a hyperbolic 3-manifold $F$ and the 2-torus. Thus, we can assume that the center of $\pi_1(M)$ has rank two. On the other hand, if $N$ has Kodaira dimension two, then it is (finitely covered by) an $S^1$ bundle over a manifold which is modeled on one of the geometries $\mathbb H^4$, $\mathbb H^2(\mathbb C)$ or $\mathbb H^2\times\mathbb H^2$. In particular, the center of $\pi_1(N)$ is infinite cyclic. Then the non-domination $M\ngeq N$ follows by a factorization argument and the asphericity of the involved spaces (cf. Lemma \ref{l:factorizationbundles}).
\end{proof}

\subsection{Kodaira dimension beyond geometries and the simplicial volume}

The notion of geometric Kodaira dimension defined here goes well beyond Thurston's geometries. This has already been explained for $3$-manifolds in the previous paragraph (Table \ref{t:3dKodaira}). Below we give a complete classification for fiber bundles in dimension four:

\begin{thm}\label{t:fiberbundlesdim4}
Let $M$ be a $4$-manifold which is (finitely covered by) a fiber bundle with fiber $F$ and base $B$. Then 
\begin{equation*}\label{eq.surfacebundles}
\kappa^g(M)= \left\{\begin{array}{ll}
        -\infty, & \text{if one of } F,B  \text{ is } S^2  \text{ or finitely covered by } \#_{m}S^2\times S^1;\\
        \ \ 0, & \text{if } F=B=T^2, \text{ or one of } F,B \text{ is a 3-manifold which is not finitely}\\
        & \text{covered by }\#_{m}S^2\times S^1  \text{ and contains no } \mathbb H^2\times\R,\widetilde{SL_2}  \text{ or } \mathbb H^3  \text{ pieces in its}\\
        &  \text{torus or sphere decomposition};\\
         \ \ 1, & \text{if one of } F,B  \text{ is } T^2  \text{ and the other is a hyperbolic surface, or one of } F,B\\
         &  \text{is a }
          \text{3-manifold with at least one } \mathbb H^2\times\R \text{ or } \widetilde{SL_2} \text{ piece and no } \mathbb H^3  \text{ pieces}\\
         &  \text{in its torus or sphere decomposition};\\
          \ \ \frac{3}{2}, & \text{if one of } B,F  \text{ is a 3-manifold with at least one } \mathbb H^3  \text{ piece in its torus or}\\
        &  \text{sphere decomposition};\\
        \ \ 2, & \text{if both } F \text{ and } B \text{ are hyperbolic surfaces}.
        \end{array}\right.
\end{equation*}
\end{thm}
\begin{proof}
The claim follows from our axioms in Definition \ref{d:gK} and $\kappa^g$ in lower dimensions.
\end{proof}

The values of the Kodaira dimension for manifolds of dimension $\leq3$ and for geometric $4$-manifolds suggest that top Kodaira dimension is often equivalent to the positivity of the simplicial volume. This is again the case for the $4$-dimensional fibrations of Theorem \ref{t:fiberbundlesdim4}:

\begin{thm}.\label{t:fiberbundlesdim4sv}
Let $M$ be a $4$-manifold which is (finitely covered by) a fiber bundle with fiber $F$ and base $B$. Then $\|M\|>0$ if and only if $\kappa^g(M)=2$.\index{simplicial volume}
\end{thm}
\begin{proof}
This is a consequence of Theorem  \ref{t:fiberbundlesdim4} and \cite[Corollary 1.3]{BN}.
\end{proof}

A natural problem stemming from this study is to understand the relationship or compatibility of the geometric Kodaira dimension $\kappa^g$ with existing notions of Kodaira dimension (see also Remark \ref{comparingKodairasol}). Motivated by the above discussion, we give the following result about the holomorphic Kodaira dimension $\kappa^h$ for complex $2n$-manifolds, which verifies  the relationship to the simplicial volume. 

\begin{thm}\cite[Theorem 1.5]{NZ}\label{t:posKod}
\begin{enumerate}
\item If $M$ is a smooth complex projective $n$-fold with non-vanishing simplicial volume, then $\kappa^h(M)\neq n-1$, $n-2$ or $n-3$.\index{Kodaira dimension!holomorphic}
\item If $M$ is a smooth K\"ahler $3$-fold with non-vanishing simplicial volume, then $\kappa^h(M)=3$.
\end{enumerate}
\end{thm}
\begin{proof}
We will summarise the main steps of the proof, giving a uniform treatment for both parts of the theorem, and refer the reader to~\cite[Section 4]{NZ} for the details. 

If $\kappa^h(M) > 0$, then $M$ admits an Iitaka fibration, namely, $M$ is birationally equivalent to a projective manifold $X$ that admits an algebraic fiber space structure $\phi\colon  X\to Y$ over a normal projective variety $Y$, such that the Kodaira dimension of a very general fiber of $\phi$ has Kodaira dimension zero. In dimensions $\leq3$, Koll\'ar~\cite{Kol95} showed that the fundamental group of a smooth proper variety with vanishing holomorphic Kodaira dimension is virtually Abelian (and conjectured that this is true in all dimensions). Using this and Gromov's mapping theorem~\cite{Gromov}, one can show that $\|X\|=0$, whenever $\dim(M)=n-3,n-2$ or $n-1$; see~\cite[Theorems 4.5 and 4.6]{NZ}. But the simplicial volume is a birational invariant~\cite[Lemma 4.1]{NZ}, hence we obtain that $\|M\|=\|X\|=0$ if $\dim(M)=n-3,n-2$ or $n-1$.

If $\kappa^h(M)=0$ and $n\leq3$, then  $\pi_1(M)$ is virtually Abelian as mentioned above, hence $M$ has vanishing simplicial volume.

Finally, we note that $\|M\|>0$ implies that $M$ cannot be  uniruled \cite[Prop. 4.2]{NZ}. Uniruled manifolds satisfy  $\kappa^h=-\infty$. In fact, Mumford conjectured that a smooth projective variety is uniruled if and only if $\kappa^h=-\infty$~\cite{BDPP}, and this is known to be true for complex projective $3$-folds~\cite{Mor}.

The proof for the K\"ahler case follows by the fact that any compact K\"ahler manifold of complex dimension three is bimeromorphic to a K\"ahler manifold which is deformation equivalent to a projective manifold~\cite{CHL,Lin}.
\end{proof}

\section{Anosov diffeomorphisms}

The final section of this survey has its origins in Dynamics and the Anosov-Smale conjecture\index{conjecture!Anosov-Smale}. Our goal is to show that Anosov diffeomorphisms\index{diffeomorphism!Anosov}\index{Anosov diffeomorphism} do not exist on geometric $4$-manifolds which are not finitely covered by the product of two aspherical surfaces. The main reference for this section is~\cite{NeoAnosov-Thurston}.

\subsection{The main result}

We will prove the following:

\begin{thm}\label{t:main}
If $M$ is a $4$-manifold that carries a geometry other than $\R^4$, $\mathbb H^2\times\R^2$ or the reducible $\mathbb H^2\times\mathbb H^2$ geometry\index{Thurston geometry}\index{manifold!geometric}, then $M$ does not admit transitive Anosov diffeomorphisms\index{Anosov diffeomorphism!transitive}.
\end{thm}

The transitivity assumption in the above theorem is mild and will only be used when $M$ is a product of the 2-sphere with an aspherical surface.

Note that Theorem \ref{t:main} does not exclude (transitive) Anosov diffeomorphisms on geometric 4-manifolds which are finitely covered by a product of surfaces $\Sigma_g\times\Sigma_h$, where $g,h\geq 1$:

\begin{prob}\cite[Section 7.2]{GL}\label{p:GL}
Does the product of two aspherical surfaces at least one of which is hyperbolic admit an Anosov diffeomorphism?
\end{prob}

Recently, D. Zhang~\cite{DZhang} showed how to exclude Anosov diffeomorphisms on certain products of two hyperbolic surfaces.

Finally, as we have done in our study thus far, we will make extended use of properties of finite covers of geometric $4$-manifolds. We thus collect the following lemmas (see~\cite{GL} and~\cite{GH} respectively):

\begin{lem}
\label{l:pre1}
Let $M$ be a manifold and $p\colon\overline M\to M$ be a finite covering. If $f\colon M\to M$ is a diffeomorphism, then there exists some integer $m\geq 0$ such that $f^m$ lifts to a diffeomorphism $\overline{f^m}\colon\overline M\to \overline M$, that is, the following diagram commutes.
$$
\xymatrix{
\overline M\ar[d]_{p} \ar[r]^{\overline{f^m}}&  \ar[d]^{p} \overline M\\
M\ar[r]^{f^m}& M \\
}
$$
\end{lem}

\begin{lem}
\label{l:pre2}
If $f\colon M\to M$ is a transitive Anosov diffeomorphism and $\overline f\colon\overline M\to\overline M$ is a lift of $f$ for some cover $\overline M$ of $M$, then $\overline f$ is transitive.
\end{lem}

\subsection{Proof of Theorem \ref{t:main}} We will examine each of the geometries and exploit tools from Algebraic Topology, such as Hirsch and Ruelle-Sullivan classes, as well as coarse geometric properties, such as negative curvature.

\subsubsection{Hyperbolic geometries}\label{s:hyperbolic}
We first deal with the real and complex hyperbolic geometries. The following theorem is well-known to experts, but we will present a proof for two reasons: First, the proof contains some useful facts about Anosov diffeomorphisms which will be used below as well, such as the behaviour of their Lefschetz numbers. Second, the tools used for the proof (e.g., the Gromov norm) reveal the beauty of connections between domains that initially might seem irrelevant to each other. 

\begin{thm}\cite{Yano,GL}\label{t:finiteout}
If $M$ is a negatively curved manifold, then $M$ does not admit Anosov diffeomorphisms. In particular, there are no Anosov diffeomorphisms on manifolds modeled on the geometry $\mathbb{H}^4$ or the geometry $\mathbb{H}^2(\mathbb{C})$.
\end{thm}
\begin{proof}
The first proof due to Yano~\cite{Yano} shows that there are no transitive Anosov diffeomorphisms on a negatively curved manifold $M$. Suppose the contrary, and let $f\colon M\to M$ be a transitive Anosov diffeomorphism. Since only tori admit codimension one Anosov diffeomorphisms~\cite{Fr,Newhouse}, we can assume that the dimension of $M$ is at least four and the codimension $k$ of $f$ is at least two. By Theorem \ref{t:RS}, there is a homology class $a\in H_l(M;\R)$ such that $H_l(f)(a)=\lambda\cdot a$ for some $\lambda>1$, where $l=k>1$ or $l=\dim(M)-k>1$. This means that the Gromov norm of $a$ is zero which is impossible because $M$ is negatively curved~\cite{Gromov,IY}.

More generally, Gogolev-Lafont~\cite{GL} showed that there are no Anosov diffeomorphisms on a negatively curved manifold $M$ of dimension $\geq 3$, using the fact that the outer automorphism group $\mathrm{Out}(\pi_1(M))$ is finite. By the latter property and the fact that $M$ is aspherical (because it is negatively curved), we conclude that some power $f^l$ of a hypothetical Anosov diffeomorphism  $f\colon M\to M$ induces the identity on cohomology (which already implies that $f$ cannot be transitive by Theorem~\ref{t:RS} or by~\cite{Shi}).  Thus, the Lefschetz numbers\index{Lefschetz number} $\Lambda$ (that is, the sum of indices of the fixed points) of any power of $f^l$ are uniformly bounded, which contradicts the growth of periodic points of $f^l$, by the equation
\begin{equation}\label{eq.FixAnosov}
|\Lambda(f^{m})|=|\mathrm{Fix}(f^{m})| = re^{mh_{top}(f)} + o(e^{mh_{top}(f)}), \ m\geq 1,
\end{equation}
where $h_{top}(f)$ denotes the topological entropy of $f$ and $r$ is the number of transitive basic sets with entropy equal to $h_{top}(f)$; we refer to~\cite[Lemma 4.1]{GL} for further details.
\end{proof}
 
\subsubsection{Product geometries}\label{s:products} 
We split our study into three cases:
\begin{itemize}
\item[(i)] Product geometries with a compact factor; 
\item[(ii)] Aspherical geometries $\mathbb{X}^3\times\R$;
\item[(iii)] The irreducible geometry $\mathbb{H}^2\times\mathbb{H}^2$.
\end{itemize}

\subsubsection*{Products with a compact factor}  
First, we show the following:

\begin{thm}\label{t:withcomp}
There are no (transitive)\index{Anosov diffeomorphism!transitive} Anosov diffeomorphisms on a $4$-manifold that carries a geometry $S^i\times\mathbb{X}$, for $i=2,3$.
\end{thm}
\begin{proof} We will examine each of the involved geometries.

\subsubsection*{The geometry $S^2\times S^2$}
Suppose $f\colon S^2\times S^2\to S^2\times S^2$ is a diffeomorphism or, more generally, a map of degree $\pm1$. The K\"unneth formula tells us that
\[
H^2(S^2\times S^2)=(H^2(S^2)\otimes H^0(S^2))\oplus(H^0(S^2)\otimes H^2(S^2)).
\]
Let $\omega_{S^2}\times 1\in H^2(S^2)\otimes H^0(S^2)$ and $1\times\omega_{S^2}\in H^0(S^2)\otimes H^2(S^2)$ be cohomological fundamental (orientation) classes. We assume that $\deg(f)=1$, after possibly passing to $f^2$. Then, by
\[
f^*(\omega_{S^2}\times1)=a\cdot(\omega_{S^2}\times 1)+b\cdot(1\times\omega_{S^2}), \ a,b\in\Z,
\]
\[
f^*(1\times\omega_{S^2})=c\cdot(\omega_{S^2}\times 1)+d\cdot(1\times\omega_{S^2}), \ c,d\in\Z,
\]
and the naturality of the cup product, we deduce that
\begin{equation}\label{eq.S2}
ad+bc=1.
\end{equation}
Moreover, since the cup product of $\omega_{S^2}\times 1$ with itself vanishes, we have
\[
0=f^*((\omega_{S^2}\times1)\cup(\omega_{S^2}\times1))=f^*(\omega_{S^2}\times1)\cup f^*(\omega_{S^2}\times1)=2ab\cdot (\omega_{S^2\times S^2}),
\]
that is,
\begin{equation}\label{eq.S2b}
ab=0.
\end{equation}
Similarly, since $(1\times\omega_{S^2})^2=0$, we have
\begin{equation}\label{eq.S2c}
cd=0.
\end{equation}
If $a=0$, then by \eqref{eq.S2}, \eqref{eq.S2b} and \eqref{eq.S2c} we obtain $b=c=\pm 1$ and $d=0$. If $b=0$, then by the same equations we obtain $a=d=\pm1$ and $c=0$. Hence, $f$ induces the identity in cohomology, after possibly replacing $f$ by $f^2$. Therefore, the Lefschetz numbers of all powers of $f$ are uniformly bounded, and so $f$ cannot be Anosov  (cf.  \eqref{eq.FixAnosov}).

\subsubsection*{The geometry $S^2\times \R^2$}
In this case, $M$ is finitely covered by $S^2\times T^2$~\cite[Theorem 10.10]{Hillman}. Since any map from $S^2$ to $T^2$ has degree zero (see Proposition \ref{p:surfaceorder}), if $f$ is a 
diffeomorphism of $S^2\times T^2$, then 
\[
f^*(\omega_{S^2}\times1)=a\cdot(\omega_{S^2}\times 1)+b\cdot(1\times\omega_{T^2}), \ a,b\in\Z,
\]
and
\[
f^*(1\times\omega_{T^2})=d\cdot(1\times\omega_{T^2}), \ d\in\Z.
\]
As above, we may assume that $\deg(f)=1$, and so by the naturality of the cup product we obtain
\begin{equation}\label{eq.S2mixied}
ad=1.
\end{equation}
In particular, $a=d=\pm1$. Also, $b=0$ because $(\omega_{S^2}\times1)^2=0$. 

Since any manifold that admits a codimension\index{Anosov diffeomorphism!codimension} one Anosov diffeomorphism is homeomorphic to a torus, we may assume that our hypothetical Anosov $f$ has codimension two. In that case, there is a class $\alpha\in H^2(S^2\times T^2;\R)$ such that $H^2(f)(\alpha)=\lambda\cdot\alpha$ for some positive real $\lambda\neq1$; see Theorem \ref{t:RS}. Now
\[
\alpha=\xi_1\cdot(\omega_{S^2}\times1)+\xi_2\cdot(1\times\omega_{T^2}), \ \xi_1,\xi_2\in\R,
\]
and by $H^2(f)(\alpha)=\lambda\cdot\alpha$ we obtain
\begin{equation}\label{eq.S2mixed-3}
\lambda\xi_1=a\xi_1=\pm\xi_1
\end{equation}
and
\begin{equation}\label{eq.S2mixed-4}
\lambda\xi_2=d\xi_2=\pm\xi_2.
\end{equation}
If $\xi_1\neq0$, then $\lambda=\pm1$ by  \eqref{eq.S2mixed-3}, which is impossible. If $\xi_1=0$, then $\xi_2\neq0$ and \eqref{eq.S2mixed-4} yields again $\lambda=\pm1$. Hence, $S^2\times T^2$ does not admit transitive Anosov diffeomorphisms.

\subsubsection*{The geometry $S^2\times\mathbb{H}^2$}
If $M$ carries the geometry $S^2\times\mathbb{H}^2$, then it is finitely covered by an $S^2$ bundle over a hyperbolic surface $\Sigma$; see~\cite[Theorem 10.7]{Hillman}. The product case $S^2\times\Sigma$ can be treated similarly to the case of $S^2\times T^2$. Moreover, Gogolev-Rodriguez Hertz showed that a $4n$-manifold $E$, which is a fiber bundle $S^{2n}\to E\to B$, does not admit transitive Anosov diffeomorphisms~\cite[Theorem 1.1]{GH} (note that this covers as well the geometry $S^2\times \R^2$). The argument in~\cite{GH} uses again Lefschetz numbers (equation \eqref{eq.FixAnosov}) and cup products via the Gysin sequence 
\[
0\longrightarrow H^{2n}(B;\Z)\longrightarrow H^{2n}(E;\Z)\longrightarrow H^0(B;\Z)\longrightarrow0.
\]
For our purposes, $2n=2$ is the only case of interest for the codimension.

\subsubsection*{The geometry $S^3\times\R$}
Finally, if a manifold carries the geometry $S^3\times\R$, then it is finitely covered by a product $S^3\times S^1$; see~\cite[Ch. 11]{Hillman}. The latter does not admit Anosov diffeomorphisms because $H_2(S^3\times S^1)=0$ and $H_1(S^3\times S^1)=\Z$.

\medskip

The proof of Theorem \ref{t:withcomp} is now complete.
\end{proof}

\subsubsection*{Aspherical geometries $\mathbb{X}^3\times\R$} Next, we will prove the following:

\begin{thm}\label{t:asphwithR}
Let $M$ be a $4$-manifold that carries one of the geometries $\mathbb{H}^3\times\R$, $Sol^3\times\R$, $\widetilde{SL_2}\times\R$ or $Nil^3\times\R$. Then $M$ does not admit Anosov diffeomorphisms.
\end{thm}

\begin{proof}
As usual, we will treat the various geometries according to their algebraic properties.

\subsubsection*{The geometries  $Nil^3\times\R$ and $\widetilde{SL_2}\times\R$}\label{ss:Hirsch2}

Suppose $M$ is a 4-manifold that carries one of the geometries $Nil^3\times\R$ or $\widetilde{SL_2}\times\R$. In this case, $M$ is finitely covered by a product $N\times S^1$, where $N$ is a $Nil^3$ manifold or an $\widetilde{SL_2}$ manifold respectively; see Proposition \ref{t:hillmanproducts}. According to Table \ref{table:3geom}, we can moreover assume that $N$ is a non-trivial $S^1$ bundle over the $2$-torus or a hyperbolic surface respectively. We conclude that the center of $\pi_1(N\times S^1)$ has rank two. A finite power of the generator of the fiber of $N$ vanishes in $H_1(N)$, and thus, for any diffeomorphism $f\colon N\times S^1\to N\times S^1$, the generator of $H_1(S^1)$ maps to a power of itself (modulo torsion). Hence, we have in cohomology
\[
H^1(f)(1\times\omega_{S^1})=a\cdot(1\times\omega_{S^1}), \ a\in\Z.
\]
Recall that $N$ is not dominated by products by Theorem \ref{t:order3-mfds} (and Table \ref{table:3geom}). Since the degree three cohomology of $N\times S^1$ is
\[
H^3(N\times S^1)\cong H^3(N)\oplus(H^2(N)\otimes H^1(S^1)),
\]
we obtain
\[
H^3(f)(\omega_N\times1)=b\cdot(\omega_N\times1), \ b\in\Z;
\]
see the proof of Theorem 1.4 in~\cite{Neodegrees} for further details. 
Since $\deg(f)=\pm1$, we deduce that $a,b\in\{\pm1\}$. Hence, we can assume that 
\[
f^*(1\times\omega_{S^1})=1\times\omega_{S^1},
\]
after replacing $f$ by $f^2$, if necessary. We conclude that $f$ is not Anosov by Lemma \ref{l:pre1} and the next theorem, which is a generalisation of Theorem \ref{t:Hirsch4}.

\begin{thm}\cite[Theorem 1]{Hirsch}\label{t:Hirschcov}
Let $f\colon M\to M$ be an Anosov diffeomorphism and a non-trivial cohomology class $u\in H^1(M;\Z)$ such that $(f^*)^m(u)=u$, for some positive integer $m$. Then the infinite cyclic covering of $M$ corresponding to $u$ has infinite dimensional rational homology.
\end{thm}

\subsubsection*{The geometries $\mathbb{H}^3\times\R$ and $Sol^3\times\R$}

Let $M$ be a 4-manifold modeled on the $\mathbb{H}^3\times\R$ or the $Sol^3\times\R$ geometry. Then, $M$ is finitely covered by $N\times S^1$, where $N$ is a hyperbolic 3-manifold or a $Sol^3$ manifold respectively; see Proposition \ref{t:hillmanproducts}. In particular, the fundamental group $\pi_1(N\times S^1)$ has infinite cyclic center generated by the $S^1$ factor, which we denote by $\pi_1(S^1)=\langle z\rangle$. If $f\colon N\times S^1\to N\times S^1$ is a diffeomorphism, then $f_*(\langle z\rangle)=\langle z\rangle$, and therefore $H^1(f)(\omega_{S^1})=\omega_{S^1}$ (replace $f$ by $f^2$, if necessary) as above, because $N$ does not admit maps of non-zero degree from direct products (see Theorem \ref{t:order3-mfds} and Table \ref{table:3geom}). Hence, $f$ cannot be Anosov by Lemma \ref{l:pre1} and Theorem \ref{t:Hirschcov}.

\begin{rem}
When $N$ is hyperbolic, the main result of~\cite{GL} implies also that $N\times S^1$ does not admit Anosov diffeomorphisms, because $\mathrm{Out}(\pi_1(N))$ is finite and $\pi_1(N)$ is Hopfian with  trivial intersection of its maximal nilpotent subgroups. In fact, as shown in~\cite{NeoAnosov2}, the only properties needed to exclude Anosov diffeomorphisms on $N\times S^1$ is that $\mathrm{Out}(\pi_1(N))$ is finite and $\pi_1(N)$ has trivial center.
\end{rem}

We have now finished the proof of Theorem \ref{t:asphwithR}.
\end{proof}

\subsubsection*{The irreducible $\mathbb{H}^2\times\mathbb{H}^2$ geometry.} Finally, suppose that $M$ carries the irreducible geometry $\mathbb{H}^2\times\mathbb{H}^2$. Then $\pi_1(M)$ has finite outer automorphism group by the strong rigidity of Mostow, Prasad and Margulis. Thus, the same argument as in the proof of Theorem \ref{t:finiteout} implies that $M$ does not admit Anosov diffeomorphisms.

\subsubsection{Non-product, solvable or compact geometries}\label{s:solvandcomp}

Finally, we prove the following:

\begin{thm}\label{t:solcomp}
There are no Anosov diffeomorphisms on a manifold carrying one of the geometries $Nil^4$, $Sol^4_{m \neq n}$, $Sol^4_0$, $Sol^4_1$, $S^4$ or $\CP^2$.
\end{thm}
\begin{proof} The proof will be done according to certain properties of the involved geometries.

\subsubsection*{The geometry $Nil^4$.}\label{ss:Nil}

Let $M$ be a 4-manifold modeled on the geometry $Nil^4$. By Proposition \ref{p:nil4}, the fundamental group of $M$ (after possibly replacing $M$ by a finite cover) is given by
\[
 \pi_1(M) = \langle x,y,z,t \ \vert \ txt^{-1}=x, \ tyt^{-1}=x^kyz^l, \ tzt^{-1} = z, [x,y]=z, \ xz=zx, \ yz=zy \rangle,
\]
where $k\geq 1$, $l\in\Z$, and it has infinite cyclic center $C(\pi_1(M)) = \langle z \rangle$. Then
\[
 \pi_1(M)/\langle z\rangle = \langle x,y,t \ \vert \ [t,y]=x^k, \ xt=tx, \ xy=yx \rangle.
\]
Let $f\colon M\to M$ be a diffeomorphism. The automorphism $f_*\colon\pi_1(M)\to\pi_1(M)$ induces an automorphism on the quotient $\pi_1(M)/\langle z\rangle$, because $f_*(\langle z\rangle)=\langle z\rangle$.  But $C( \pi_1(M)/\langle z\rangle)=\langle x\rangle$, hence $f_*(x)=z^nx^m$, for some $n,m\in\Z$, $m\neq0$. Since $[x,y]=z$ and by the fact that the relation $txt^{-1}=x$ is mapped to $f_*(t)x^mf_*(t)^{-1}=x^m$, we obtain that $f_*(t)$ does not contain any powers of $y$. 
By the commutative diagram
$$
\xymatrix{
\pi_1(M)\ar[d]_{h} \ar[r]^{{f}_*}&  \ar[d]^{h}  \pi_1(M)\\
H_1(M;\Z) \ar[r]^{{H_1(f)}}& H_1(M;\Z), \\
}
$$
where 
\[
h\colon\pi_1(M)\to H_1(M;\Z)=\pi_1(M)/[\pi_1(M),\pi_1(M)]
\]
denotes the Hurewicz homomorphism, we conclude that $H_1(f)$ maps the homology class $\bar t\in H_1(M;\Z)/Tor H_1(M;\Z)$ to a multiple of itself. In fact, the induced automorphism on $H_1(M;\Z)/Tor H_1(M;\Z)=\langle \bar t\rangle\times \langle \bar y\rangle=\Z\times\Z$ implies that $H_1(f)(\bar t)=\bar t$, which means that $f$ cannot be Anosov by Lemma \ref{l:pre1} and Theorem \ref{t:Hirsch4} (or Theorem \ref{t:Hirschcov}).

\subsubsection*{The geometries $Sol_{m \neq n}^4$, $Sol^4_0$ and $Sol^4_1$}

For these geometries, the following immediate consequence of Theorem \ref{t:Hirsch4} will suffice:

\begin{thm}\cite[Theorem 8]{Hirsch}\label{t:Hirsch}
Suppose $M$ is a manifold such that 
\begin{itemize}
\item[(a)] $\pi_1(M)$ is virtually polycyclic;
\item[(b)] the universal covering of $M$ has finite dimensional rational homology;
\item[(c)] $H^1(M;\Z)\cong\Z$.
\end{itemize}
Then $M$ does not admit Anosov diffeomorphisms.
\end{thm}

Suppose $M$ is a 4-manifold modeled on one among the geometries $Sol^4_{m \neq n}$, $Sol^4_0$ or $Sol^4_1$. If $M$ carries one of the first two geometries, then $M$ is a mapping torus of a hyperbolic homeomorphism of $T^3$; see Theorem \ref{t:mappingtorisolvable}(1); in particular, $H^1(M;\Z)\cong\Z$. Since $\pi_1(M)$ is polycyclic and $M$ is aspherical, Lemma \ref{l:pre1} and Theorem \ref{t:Hirsch} imply that $M$ does not admit Anosov diffeomorphisms. 
If $M$ carries the geometry $Sol_1^4$, then, by Theorem \ref{t:mappingtorisolvable}(2) (see also Proposition~\ref{p:sol1}), we have
\begin{eqnarray*}
  \pi_1(M) = \langle x,y,z,t \ | \ txt^{-1}=x^ay^cz^k,  tyt^{-1}=x^by^dz^l,  tzt^{-1} =z,
             [x,y]=z, xz=zx,  yz=zy \rangle,
\end{eqnarray*}
where $k,l\in\Z$ and the matrix
\[
\left(\begin{array}{cc}
   a & b \\
   c & d \\
\end{array} \right)
\]
 has no eigenvalues which are roots of unity. 
The Abelianization of $\pi_1(M)$ tells us that $H^1(M;\Z)\cong\Z$. Since moreover $M$ is aspherical and $\pi_1(M)$ is polycyclic, Lemma \ref{l:pre1} and Theorem \ref{t:Hirsch} imply that $M$ does not admit Anosov diffeomorphisms.
 
\begin{rem}
Note that Theorem \ref{t:Hirsch} is not applicable if $M$ is a $Nil^4$ manifold, because $H^1(M;\Z)\cong\Z^2$.
\end{rem}

\subsubsection*{The geometries $S^4$ and $\CP^2$}
The only 4-manifold modeled on $S^4$ is $S^4$ itself~\cite[Section 12.1]{Hillman}. Since any orientation preserving diffeomorphism $f\colon S^4\to S^4$ induces the identity in homology,  $f$ cannot be Anosov (cf. \eqref{eq.FixAnosov}).

Similarly to $S^4$, the only 4-manifold modeled on $\CP^2$ is $\CP^2$ itself~\cite[Section 12.1]{Hillman}. Let $f\colon \CP^2\to \CP^2$ 
 be a diffeomorphism. Since the cohomology groups of $\CP^2$ are $\Z$ in degrees 0, 2 and 4 and trivial otherwise, $f^m$ induces the identity on cohomology, for some $m\geq1$. Hence, $f$ cannot be Anosov.
 
 \medskip
 
We have now completed the proof of Theorem \ref{t:solcomp}.
 \end{proof}

This finishes the proof of Theorem \ref{t:main}.

\include{index}

 \printindex


\begin{thebibliography}{10}

\bibitem{Agol}
I. Agol, {\em The virtual Haken conjecture}, With an appendix by I. Agol, D. Groves and J. Manning, Doc. Math. {\bf 18} (2013), 1045--1087.

\bibitem{BDPP}
S. Boucksom, J.-P. Demailly, M. P\u aun and T. Peternell, {\em The pseudo-effective cone of a compact K\"ahler manifold and varieties of negative Kodaira dimension}, J. Algebraic Geom. {\bf 22} (2013), 201--248.

\bibitem{BG}
R. Brooks and W. Goldman, {\em Volumes in Seifert space}, Duke Math. J. {\bf 51} (1984), 529--545.

\bibitem{Bucher}
M. Bucher, {\em Simplicial volume of locally symmetric spaces covered by $\mathrm{SL}_{3}\mathbb{R}/\mathrm{SO}(3)$}, Geom. Dedicata {\bf 125} (2007), 203--224.

\bibitem{Buch}
M. Bucher, {\em The simplicial volume of closed manifolds covered by $\mathbb{H}^2\times\mathbb{H}^2$}, J. Topol. {\bf 1} (2008), 584--602.

\bibitem{BN}
M. Bucher and C. Neofytidis, {\em The simplicial volume of mapping tori of 3-manifolds}, Math. Ann. {\bf 376} (2020), 1429--1447.

\bibitem{CT}
J. A. Carlson and D. Toledo, {\em Harmonic mappings of K\"ahler manifolds to locally symmetric spaces}, Inst. Hautes \'Etudes Sci. Publ. Math. {\bf 69} (1989), 173--201.
 
 \bibitem{CHL}
 B. Claudon, A. H\"oring and H.-Y. Lin, {\em The fundamental group of compact K\"ahler threefolds}, Geom. Topol. {\bf 23} (2019), 3233--3271.
 
 \bibitem{CKS}
 C. Cassidy, N. Kennedy and D. Scevenels, {\em Hyperbolic automorphisms for groups in $T(4,2)$}, Crystallographic groups and their generalizations (Kortrijk, 1999), 171--175, 
Contemp. Math. {\bf 262}, Amer. Math. Soc., Providence, RI, 2000. 
 
 \bibitem{CL}
 D. Crowley and C. L\"oh, {\em Functorial seminorms on singular homology and (in)flexible manifolds}, Algebr. Geom. Topol. {\bf 15} (2015), 1453--1499.

\bibitem{Filipkiewicz}
R. Filipkiewicz, {\sl Four-dimensional geometries}, PhD thesis, University of Warwick, 1983.

\bibitem{Fr}
J. Franks, {\sl Anosov Diffeomorphisms. Global Analysis}, Proceedings of Symposia in Pure Mathematics, pp. 61--93, AMS, Providence, R.I 1970.

\bibitem{Gaifullin}
A. Gaifullin, {\em Universal realisators for homology classes}, Geom. Top. {\bf 17} (2013), 1745--1772.

\bibitem{Geng1}
A. Geng, {\em 5-dimensional geometries I: the general classification}, Preprint: arXiv:1605.07545.

\bibitem{Geng2}
A. Geng, {\em 5-dimensional geometries II: the non-fibered geometries}, Preprint: arXiv:1605.07534.

\bibitem{Geng3}
A. Geng, {\em 5-dimensional geometries III: the fibered geometries}, Preprint: arXiv:1605.07546.

\bibitem{GH}
A. Gogolev and F. Rodriguez Hertz, {\em Manifolds with higher homotopy which do not support Anosov diffeomorphisms}, Bull. Lond. Math Soc. {\bf 46} (2014), 349--366.

\bibitem{GL}
A. Gogolev and J.-F. Lafont, {\em Aspherical products which do not support Anosov diffeomorphisms}, Ann. Henri Poincar\'e {\bf 17} (2016), 3005--3026.

\bibitem{Gromov}
M. Gromov, {\em Volume and bounded cohomology}, Inst. Hautes \'Etudes Sci. Publ. Math. {\bf 56} (1982), 5--99.

\bibitem{Gromovbook}
M.~Gromov, {\sl Metric Structures for Riemannian and Non-Riemannian Spaces}, with appendices by M.~Katz, P.~Pansu and S.~Semmes, 
translated from the French by S.~M.~Bates, Progress in Math. {\bf 152}, Birkh\"auser Verlag, 1999.

\bibitem{Gromovessay}
M. Gromov, {\em Hyperbolic groups}, in “Essays in Group Theory”, Math. Sci. Res. Inst. Publ., Springer, New York-Berlin {\bf 8} (1987), 75--263.

\bibitem{Ha}
A. Hatcher, {\sl Algebraic topology}, Cambridge University Press, Cambridge, 2002.

\bibitem{Hillman}
J. A. Hillman, {\sl Four-manifolds, geometries and knots}, Geom. Topol. Monographs {\bf 5}, Coventry, 2002.

\bibitem{Hirsch}
M. W. Hirsch, {\em Anosov maps, polycyclic groups, and homology}, Topology {\bf 10} (1971), 177--183.

\bibitem{IY}
H. Inoue and K. Yano, {\em The Gromov invariant of negatively curved manifolds}, Topology {\bf 21} no. 1 (1981), 83--89.

\bibitem{Kirby}
R. Kirby, {\sl Problems in low-dimensional topology}, Berkeley 1995.

\bibitem{Kol95} 
J. Koll\'ar, {\sl Shafarevich maps and automorphic forms}, M. B. Porter Lectures. Princeton University Press, Princeton, NJ, 1995, x+201 pp.

\bibitem{Kotschick:4-mfds}
D. Kotschick, {\em Remarks on geometric structures on compact complex surfaces}, Topology {\bf 31} (1992), 317--321.

\bibitem{KL}
D. Kotschick and C. L\"oh, {\em Fundamental classes not representable by products}, J. London Math. Soc. {\bf 79} (2009), 545--561.

\bibitem{KotschickLoehNeofytidis}
D. Kotschick, C. L\"oh and C. Neofytidis, {\em On stability of non-domination under taking products}, Proc. Amer. Math. Soc. {\bf 144} (2016), 2705--2710.

\bibitem{KN}
D. Kotschick and C. Neofytidis, {\em On three-manifolds dominated by circle bundles}, Math. Z. {\bf 274} (2013), 21--32.

\bibitem{LS}
J.-F.~Lafont and B.~Schmidt, {\em Simplicial volume of closed locally symmetric spaces of non-compact type}, Acta Math. {\bf 197} (2006), 129--143.

\bibitem{Lisurvey}
T.-J. Li, {\em Kodaira dimension in low-dimensional topology}, Tsinghua lectures in mathematics, 265--291, Adv. Lect. Math. (ALM) {\bf45}, Int. Press, Somerville, MA, 2019.

\bibitem{Lin} 
H.-Y. Lin, {\em Algebraic approximations of  compact K\"ahler threefolds}, arXiv:1710.01083.

\bibitem{Mal}
W. Malfait, {\em Anosov diffeomorphisms on nilmanifolds of dimension at most six}, Geom. Dedicata {\bf 79} (2000), 291--298.


\bibitem{Margulis}
G. A. Margulis, {\sl Discrete subgroups of semisimple Lie groups}, Ergebnisse der Mathematik und ihrer Grenzgebiete, 3. Folge Bd. 17. Springer-Verlag, Berlin-Heidelberg, 1991.

\bibitem{Mineyev}
I. Mineyev, {\em Bounded cohomology characterizes hyperbolic groups}, Q. J. Math. {\bf 53} (2002), 5--73.

\bibitem{Mor}
S. Mori, {\em Flip theorem and the existence of minimal models for $3$-folds}, J. Amer. Math. Soc. {\bf 1} (1988), 117--253. 

\bibitem{Neothesis}
C. Neofytidis, {\sl Non-zero degree maps between manifolds and groups presentable by products}, Munich thesis, 2014, available online at http://edoc.ub.uni-muenchen.de/17204/.

\bibitem{Neobranch}
C. Neofytidis, {\em Branched coverings of simply connected manifolds}, Topol. Appl. {\bf 178} (2014), 360--371.

\bibitem{Neodegrees}
C.~Neofytidis, {\em Degrees of self-maps of products}, Int. Math. Res. Not. IMRN {\bf 22} (2017), 6977-6989.

\bibitem{NeoIIPP}
C.~Neofytidis, {\em Fundamental groups of aspherical manifolds and maps of non-zero degree}, Groups Geom. Dyn. {\bf 12} (2018), 637--677.

\bibitem{Neoorder}
C.~Neofytidis, {\em Ordering Thurston's geometries by maps of non-zero degree}, J. Topol. Anal. {\bf 10} (2018), 853--872.

\bibitem{NeoAnosov1}
C. Neofytidis, {\em Anosov diffeomorphisms of products I. Negative curvature and rational homology spheres}, Ergodic Theory Dynam. Systems {\bf 41} (2021), 553--569. 

\bibitem{NeoAnosov2}
C. Neofytidis, {\em Anosov diffeomorphisms of products II. Aspherical manifolds}, J. Pure Appl. Algebra {\bf 224} no. 3 (2020), 1102--1114.

\bibitem{NeoAnosov-Thurston}
C.~Neofytidis, {\em Anosov diffeomorphisms on Thurston geometric 4-manifolds}, Geom. Dedicata {\bf 213} (2021), 325--337.

\bibitem{NeoHopf}
C.~Neofytidis, {\em On a problem of Hopf for circle bundles over aspherical manifolds with hyperbolic fundamental group}, Algebr. Geom. Topol. {\bf 23} (2023), 3205--3220.

\bibitem{NZ}
C. Neofytidis and W. Zhang, {\em Geometric structures, the Gromov order, Kodaira dimensions and simplicial volume}, Pacific J. Math. {\bf 315} (2021), 209--233.

\bibitem{Newhouse}
S. E. Newhouse, {\em On codimension one Anosov diffeomorphisms}, Amer. J. Math. {\bf 92} (1970), 761--770.

\bibitem{RS}
D. Ruelle and D. Sullivan, {\em Currents, flows and diffeomorphisms}, Topology {\bf 14} (1975), 319--327.

\bibitem{Scott}
P. Scott, {\em The geometries of 3-manifolds}, Bull. London Math. Soc. {\bf 15} (1983), 401--487.

\bibitem{Shi}
K. Shiraiwa, {\em Manifolds which do not admit Anosov diffeomorphisms}, Nagoya Math. J. {\bf 49} (1973), 111--115.

\bibitem{Sm}
S.~Smale, {\em Differentiable dynamical systems}, Bull. Am. Math. Soc. {\bf 73} (1967), 747--817.

\bibitem{Thom}
R. Thom, {\em Quelques propri\'et\'es globales des vari\'et\'es diff\'erentiables}, Comment. Math. Helv. {\bf 28} (1954), 17--86.

\bibitem{Thu}
W. P. Thurston, {\em Three-Dimensional Geometry and Topology}, Princeton University Press, 1997.

\bibitem{Wall1}
C. T. C. Wall, {\em Geometries and geometric structures in real dimension 4 and complex dimension 2}, In Geometry and Topology (College Park, Md., 1983/84), Lecture Notes in Math. {\bf 1167}, Springer, Berlin (1985), 268--292.

\bibitem{Wall2}
C. T. C. Wall, {\em Geometric structures on compact complex analytic surfaces}, Topology {\bf 25} (1986), 119--153.

\bibitem{Wangorder}
S. Wang, {\em The existence of maps of non-zero degree between aspherical 3-manifolds}, Math. Z. {\bf 208} (1991), 147--160.

\bibitem{Yano}
K. Yano, {\em There are no transitive Anosov diffeomorphisms on negatively curved manifolds}, Proc. Japan Acad. Ser. A Math. Sci. {\bf 59} (1983), 445.

\bibitem{DZhang}
D. Zhang, {\em Anosov diffeomorphisms on a product of surfaces}, Preprint: arXiv:2205.11296.

\bibitem{Zhang}
W. Zhang, {\em Geometric structures, Gromov norm and Kodaira dimensions}, Adv. Math. {\bf 308} (2017), 1--35. 

\end{thebibliography}
\end{document}